\documentclass[a4paper,11pt,fullpage]{article}
\usepackage[english]{babel}

\usepackage[left=1in,right=1in,top=1in,bottom=1in]{geometry}

\usepackage{amsmath}
\usepackage{amsthm}
\usepackage{amsfonts}
\usepackage{amssymb}
\usepackage{amsxtra}
\usepackage{latexsym}
\usepackage{ifthen}
\usepackage{cite}
\usepackage{esint}
\usepackage[cp1250]{inputenc}

\usepackage{epic}
\usepackage{eepic}

\DeclareMathOperator*{\esssup}{ess\,sup}

\DeclareMathOperator*{\essinf}{ess\,inf}

\DeclareMathOperator*{\osc}{osc}

\numberwithin{equation}{section}
\newtheorem{theorem}{Theorem}[section]
\newtheorem{lemma}{Lemma}[section]
\newtheorem{remark}{Remark}[section]
\newtheorem{definition}{Definition}[section]

\def\XXint#1#2#3{{\setbox0=\hbox{$#1{#2#3}{\int}$}
     \vcenter{\hbox{$#2#3$}}\kern-.5\wd0}}

\begin{document}

\title{On the continuity of solutions of quasilinear parabolic equations with generalized
Orlicz growth under non-logarithmic conditions
%\thanks{Dedicated to the 80th anniversary of Professor Igor V. Skrypnik.}
}

\author{Igor I. Skrypnik, Mykhailo V. Voitovych
%\thanks{The Division of Applied Problems
 %in Contemporary Analysis, Institute of Mathematics of NASU, Kiev, Ukraine}
 }

  \maketitle

\begin{abstract}
We prove the continuity of bounded solutions for a wide class of parabolic
equations with $(p,q)$-growth
$$
u_{t}-{\rm div}\left(g(x,t,|\nabla u|)\,\frac{\nabla u}{|\nabla u|}\right)=0,
$$
under the generalized non-logarithmic Zhikov's condition
$$
g(x,t,{\rm v}/r)\leqslant c(K)\,g(y,\tau,{\rm v}/r),
\quad (x,t), (y,\tau)\in Q_{r,r}(x_{0},t_{0}),
\quad 0<{\rm v}\leqslant K\lambda(r),
$$
$$
\quad \lim\limits_{r\rightarrow0}\lambda(r)=0,
\quad \lim\limits_{r\rightarrow0} \frac{\lambda(r)}{r}=+\infty,
\quad \int_{0} \lambda(r)\,\frac{dr}{r}=+\infty.
$$
In particular, our results cover new cases of double-phase
parabolic equations.

% We prove the continuity of bounded solutions to
% parabolic equations of the type
%$$
%\begin{aligned}
%u_{t}-{\rm div}\big(|\nabla u|^{p-2}\,\nabla u+a(x,t)\,|\nabla u|^{q-2}\,\nabla u\big)=0,
%\quad &a(x,t)\geqslant0,
%\\
%u_{t}-{\rm div}\big(|\nabla u|^{p-2}\,\nabla u+a(x,t)\,|\nabla u|^{q-2}\nabla u
%[1+\ln(1+|\nabla u|)]^{\beta}  \big)=0,
%\quad &a(x,t)\geqslant0, \ \ \beta\geqslant0,
%\\
%u_{t}-{\rm div}\left(|\nabla u|^{p-2}\,\nabla u
%\big[1+\ln(1+b(x,t)\,|\nabla u|)\big]\right)=0,
%\quad &b(x,t)\geqslant0
%\end{aligned}
%$$
%under the precise choice of $a(x,t)$ and $b(x,t)$.



\end{abstract}

\pagestyle{myheadings} \thispagestyle{plain}
\markboth{Igor I. Skrypnik}
{On the continuity of solutions of quasilinear parabolic equations . . .}

\section{Introduction and main results}\label{Introduction}
In this paper we are concerned with a class of parabolic equations with
non-standard growth conditions.
Let $\Omega$ be a bounded domain in $\mathbb{R}^{n}$, $n\geqslant2$, $T>0$,
$\Omega_{T}:=\Omega\times(0,T)$.
We study solutions to the equation
\begin{equation}\label{eq1.1}
u_{t}-{\rm div} \mathbf{A}(x, t, \nabla u)=0, \quad (x,t)\in\Omega_{T}.
\end{equation}
Throughout the paper we suppose that the functions
$\mathbf{A}:\Omega_{T}\times\mathbb{R}^{n}\rightarrow\mathbb{R}^{n}$
are such that $\mathbf{A}(\cdot,\cdot,\xi)$ are Lebesgue measurable for all
$\xi\in \mathbb{R}^{n}$, and $\mathbf{A}(x,t,\cdot)$ are continuous for almost all
$(x,t)\in\Omega_{T}$.
We assume also that the following structure conditions are satisfied
\begin{equation}\label{eq1.2}
\begin{aligned}
\mathbf{A}(x,t,\xi)\,\xi&\geqslant K_{1}\,g(x,t,|\xi|)\,|\xi|,
\\
|\mathbf{A}(x,t,\xi)|&\leqslant K_{2}\,g(x,t,|\xi|),
\end{aligned}
\end{equation}
where $K_{1}$, $K_{2}$ are positive constants, and
$g(x,t,{\rm v})$, ${\rm v}>0$,
is a positive function, satisfying conditions which will be specified below.

The aim of this paper is to establish the continuity of bounded weak solutions
for a wide class of parabolic equations with generalized Orlicz growth. In terms
of the function $g$, this class can be characterized as follows. Let
$g(x,t, {\rm v}):\Omega_{T}\times\mathbb{R}_{+}\rightarrow\mathbb{R}_{+}$
be a non-negative function satisfying the following properties: for a.a. $(x,t)\in\Omega_{T}$,
the function ${\rm v}\rightarrow g(x,t,{\rm v})$ is increasing, and
$\lim\limits_{{\rm v}\rightarrow0}g(x,t,{\rm v})=0$,
$\lim\limits_{{\rm v}\rightarrow+\infty}g(x,t,{\rm v})=+\infty$.

In addition, we assume that
 \begin{itemize}
\item[(${\rm g}_{1}$)]
there exist $1<p<q$ such that
for $(x,t)\in \Omega_{T}$ and for ${\rm w}\geqslant{\rm v}> 0$ there hold
\begin{equation*}\label{gqineq}
\left( \frac{{\rm w}}{{\rm v}} \right)^{p-1}
\leqslant\frac{g(x,t, {\rm w})}{g(x,t, {\rm v})}\leqslant
\left( \frac{{\rm w}}{{\rm v}} \right)^{q-1};
\end{equation*}
\end{itemize}
 \begin{itemize}
\item[(${\rm g}_{2}$)]
for any $K, K_{3}>0$ and any cylinder $Q_{R_{0}, R_{0}}(x_{0},t_{0})\subset\Omega_{T}$,
there exist
$c_{1}(K,K_{3})>0$ and positive continuous and non-decreasing function $\lambda(r)$ on the interval
$(0,R_{0})$, $\lim\limits_{r\rightarrow0}\lambda(r)=0$,
$\lim\limits_{r\rightarrow0} r^{1-\delta_{0}}/\lambda(r)=0$ and
$\lambda(2r)\leqslant (3/2)^{1-\delta_{0}}\lambda(r)$, $0<r<R/2$,
with some $\delta_{0}\in(0,1)$, such that
$$
g(x,t, {\rm v}/r)\leqslant c_{1}(K,K_{3})\, g(y,\tau, {\rm v}/r)
$$
for any $(x,t)$, $(y,\tau)\in Q_{r,K_{3}r}(x_{0},t_{0})\subset Q_{R_{0}, R_{0}}(x_{0},t_{0})$
and for all $0<{\rm v}\leqslant K\lambda(r)$.
Here
$Q_{R_{1}, R_{2}}(x_{0},t_{0}):=B_{R_{1}}(x_{0})\times (t_{0}-R_{2},t_{0})$,
$B_{R_{1}}(x_{0}):=\{x\in \mathbb{R}^{n}: |x-x_{0}|<R_{1}\}$, $R_{1},R_{2}>0$.
\end{itemize}
\begin{remark}\label{rem1.1}
{\rm The function $g_{a(x,t)}({\rm v}):={\rm v}^{\,p-1}+a(x,t){\rm v}^{\,q-1}$,
$a(x,t)\geqslant0$, ${\rm v}>0$, where
$$
\osc\limits_{Q_{r,r}(x_{0},t_{0})}a(x,t)\leqslant A\,r^{q-p}\mu^{q-p}(r),
\ \ A>0, \ \  q>p,
$$
\begin{equation*}\label{1.3}
\lim\limits_{r\rightarrow0}\mu(r)=+\infty, \quad
\lim\limits_{r\rightarrow0}r\mu(r)=0,
\end{equation*}
satisfies condition (${\rm g}_{2}$) with $\lambda(r)=1/\mu(r)$.
Indeed, if $0<{\rm v}\leqslant K\lambda(r)$ then
\begin{multline*}
g_{a(x,t)}({\rm v}/r)-g_{a(y,\tau)}({\rm v}/r)\leqslant
|a(x,t)-a(y,\tau)|\,({\rm v}/r)^{q-1}
\\
\leqslant A\,[K\mu(r)\lambda(r)]^{q-p}\,({\rm v}/r)^{p-1}
\leqslant AK^{q-p}\,({\rm v}/r)^{p-1}\leqslant AK^{q-p}\,g_{a(y,\tau)}({\rm v}/r).
\end{multline*}
Similarly, the function
$\widetilde{g}_{a(x,t)}({\rm v}):={\rm v}^{\,p-1}+a(x,t){\rm v}^{\,\widetilde{q}-1}
[1+\ln(1+{\rm v})]^{\beta}$, $a(x,t)\geqslant0$, $\beta\geqslant0$, ${\rm v}>0$, where
$$
\osc\limits_{Q_{r,r}(x_{0},t_{0})}a(x,t)\leqslant A\,r^{\widetilde{q}-p}\mu^{\widetilde{q}-p}(r),
\ \ A>0, \ \  \widetilde{q}>p, \ \ q=\widetilde{q}+\beta,
$$
$$
\lim\limits_{r\rightarrow0}\mu(r)=+\infty, \quad
\lim\limits_{r\rightarrow0}r\mu(r)\ln^{\frac{\beta}{\widetilde{q}-p}}\frac{1}{r}=0,
$$
satisfies condition (${\rm g}_{2}$) with
$\lambda(r)= \dfrac{\ln^{-\frac{\beta}{\widetilde{q}-p}}\dfrac{1}{r}}{\mu(r)}$.

}
\end{remark}
\begin{remark}\label{rem1.2}
{\rm The function
$g_{b(x,t)}({\rm v}):={\rm v}^{\,p-1}\Big[1+\ln\big( 1+b(x,t){\rm v} \big) \Big]$,
$b(x,t)\geqslant0$, ${\rm v}>0$, where
$\osc\limits_{Q_{r,r}(x_{0},t_{0})}b(x,t)\leqslant B\,r\mu(r)$, $B>0$,
$\lim\limits_{r\rightarrow0}\mu(r)=+\infty$,
$\lim\limits_{r\rightarrow0}r\mu(r)=0$,
satisfies condition (${\rm g}_{2}$) with $\lambda(r)=1/\mu(r)$. Indeed
$$
\begin{aligned}
g_{b(x,t)}({\rm v}/r)&-g_{b(y,\tau)}({\rm v}/r)=
\left( \frac{{\rm v}}{r} \right)^{p-1}
\ln\frac{1+b(x,t)\dfrac{{\rm v}}{r}}{1+b(y,\tau)\dfrac{{\rm v}}{r}}
\\
&\leqslant \left( \frac{{\rm v}}{r} \right)^{p-1}
\ln \left( 1+ |b(x,t)-b(y,\tau)|\, \frac{{\rm v}}{r} \right)
\leqslant
\left( \frac{{\rm v}}{r} \right)^{p-1}\ln\big(1+BK\mu(r)\lambda(r)\big)
\\
&=\left( \frac{{\rm v}}{r} \right)^{p-1}\ln(1+BK)\leqslant
\ln(1+BK)\, g_{b(y,\tau)}({\rm v}/r) \ \ \text{if} \ \ 0<{\rm v}\leqslant K\lambda(r).
\end{aligned}
$$
}
\end{remark}
\begin{remark}\label{rem1.3}
{\rm Consider the functions
$$
g_{1}(x,t,{\rm v}):={\rm v}^{\,p(x,t)-1}, \ \
g_{2}(x,t,{\rm v}):={\rm v}^{\,p-1}\big( 1+C(x,t)\ln(1+{\rm v}) \big),
\ \ (x,t)\in \Omega_{T}, \ \ {\rm v}>0,
$$
where
$$
|p(x,t)-p(y,\tau)|+|C(x,t)-C(y,\tau)|
\leqslant
\frac{L}{\ln \dfrac{1}{r\mu(r)}},
\quad
(x,t), (y,\tau) \in Q_{r,r}(x_{0},t_{0}),
$$
$$
0<L<+\infty, \quad \lim\limits_{r\rightarrow0}\mu(r)=+\infty,
\quad \lim\limits_{r\rightarrow0}r^{1-\delta_{0}}\mu(r)=0, \quad \delta_{0}\in (0,1).
$$
It is obvious that the functions $g_{1}$, $g_{2}$ satisfy condition
(${\rm g}_{2}$) with $\lambda(r)=1/\mu(r)$.
But by our choices the following inequalities hold:
$\delta_{0}\ln\dfrac{1}{r}\leqslant \ln\dfrac{1}{r\mu(r)}\leqslant\ln\dfrac{1}{r}$.
So, in this case condition (${\rm g}_{2}$) is equivalent to the logarithmic Zhikov's condition.
In this case, the qualitative properties of solutions to parabolic equations are known
(see, e.g., \cite{AntZhikov2005, BarBog2014, BogDuzaar2012, DingZhangZhou2020, WinkZach2016, XuChen2006, Yao2014, Yao2015,
ZhZhouXueNonAn2014, ZhikPast2010MatNotes, ZhikAlkhTSP11}
and references therein).
}
\end{remark}

The study of regularity of minima
of functionals with non-standard growth has been initiated by Zhikov
\cite{ZhikIzv1983, ZhikIzv1986, ZhikJMathPh94, ZhikJMathPh9798, ZhikKozlOlein94},
Marcellini \cite{Marcellini1989, Marcellini1991}, and Lieberman \cite{Lieberman91},
and in the last thirty years, the qualitative theory
of second order elliptic and parabolic equations with so-called log-condition
(i.e. if $\lambda(r)\equiv1$) has been actively developed
(see, e.g.,
\cite{DienHarHastRuzVarEpn, HarHastOrlicz, HarHastLeNuorNA2010, MingioneDarkSide,
SkrVoitUMB19, SkrVoitNA20, VoitNA19} for the overviews).
Equations of this type and systems of such equations arise in various
problems of mathematical physics (see the monographs
\cite{AntDiazShm2002monogr, HarHastOrlicz, Ruzicka2000, Weickert} for the application background).

Double-phase elliptic equations under the logarithmic condition were studied by
Colombo, Mingione \cite{ColMing218, ColMing15, ColMingJFnctAn16}
and by Baroni, Colombo, Mingione \cite{BarColMing, BarColMingStPt16, BarColMingCalc.Var.18}.
Particularly,
$C^{0,\beta}_{{\rm loc}}(\Omega)$, $C^{1,\beta}_{{\rm loc}}(\Omega)$ and Harnack's
inequality were obtained in the case $\lambda(r)\equiv1$ under the precise
conditions on the parameters $\alpha$, $p$, $q$.
The continuity of solutions and Harnack's inequality for double-phase parabolic equations
were proved in \cite{BurchSkrPotAn} under the logarithmic condition (${\rm g}_{2}$)
(i.e. if $\lambda(r)=1$).

The case when condition (${\rm g}_{2}$) holds differs substantionally from the
logarithmic case. To our knowledge there are few results in this direction.
Zhikov \cite{ZhikPOMI04}
obtained a generalization of the logarithmic condition which guarantees the density
of smooth functions in Sobolev space $W^{1,p(x)}(\Omega)$. Particularly,
this result holds if $1<p\leqslant p(x)$ and
$$
|p(x)-p(y)|\leqslant L\,
\frac{ \Big|\ln \big|\ln |x-y|\big| \Big|}{\big|\ln |x-y|\big|},
\quad x,y\in\Omega, \quad x\neq y, \quad 0<L<p/n.
$$
Later Zhikov and Pastukhova \cite{ZhikPast2008MatSb}
under the same condition proved higher integrability
of the gradient of solutions to the $p(x)$-Laplace equation.

Interior continuity, continuity up to the boundary and Harnack's inequality
to $p(x)$-Laplace equation were proved by Alkhutov, Krasheninnikova \cite{AlhutovKrash08},
Alkhutov, Surnachev \cite{AlkhSurnAlgAn19} and Surnachev \cite{SurnPrepr2018} under the condition
$$
|p(x)-p(y)|\leqslant \frac{\tau(|x-y|)}{\big|\ln|x-y|\big|},
\quad x,y\in\Omega, \quad x\neq y,
$$
\begin{equation*}\label{eq1.13}
\lim\limits_{r\rightarrow0}\tau(r)=+\infty \quad \text{and} \quad
\int_{0} \exp\Big( -\gamma\exp\big(\beta \tau(r)\big) \Big)\frac{dr}{r}=+\infty,
\end{equation*}
with some constants $\gamma>0$, $\beta>1$. Particularly, the function
$\tau(r)=L\ln\ln\ln \dfrac{1}{r}$, $0<L<\dfrac{1}{\beta}$ satisfies the above
conditions. These results were generalized in \cite{SkrVoitNA20, ShSkrVoit20}
for a wide class of elliptic and parabolic equations with non-logarithmic Orlicz growth.
Particularly, it was proved in \cite{SkrVoitNA20} the interior continuity of solutions
under the condition
$$
g(x,t_{1}, {\rm v}/r)\leqslant c(K)\,\tau(r)\,g(y,t_{2}, {\rm v}/r),  \ \
(x,t_{1}), (y,t_{2})\in Q_{r,r}(x_{0},t_{0}),\ \
r\leqslant {\rm v} \leqslant K,
$$
$$
\lim\limits_{r\rightarrow0}\tau(r)=+\infty, \quad
\lim\limits_{r\rightarrow0}r^{1-\delta_{0}}\tau(r)=0, \quad
\delta_{0}\in (0,1), \quad \text{and}
$$
\begin{equation}\label{eq1.3}
\int_{0} \exp\left( -\gamma\tau^{\beta}(r) \right)\frac{dr}{r}=+\infty,
\end{equation}
with some constants $\gamma>0$, $\beta>1$.
\begin{remark}
{\rm
Note that for double-phase parabolic equations
(i.e. if $g(x,t,\cdot)=g_{a(x,t)}(\cdot)$) condition \eqref{eq1.3} is valid, if
$\tau(r)=\mu^{q-p}(r)=\left(\ln\ln\dfrac{1}{r}\right)^{L}$, $0\leqslant L<\dfrac{1}{\beta}$.
We also note that if $g(x,t,\cdot)=g_{a(x,t)}(\cdot)$ or $g(x,t,\cdot)=g_{b(x,t)}(\cdot)$ and
$\mu(r)=\ln \dfrac{1}{r}$, then condition \eqref{eq1.3} fails.
}
\end{remark}

Elliptic equations under condition (${\rm g}_{2}$) were considered in \cite{HadzhySkrVoit}.
Particularly, it was proved in \cite{HadzhySkrVoit} that for the continuity of solutions to
elliptic equations with non-logarithmic growth, it is sufficient to satisfy
\begin{equation}\label{eq1.4}
\int_{0} \lambda(r)\, \frac{dr}{r}=+\infty.
\end{equation}
It turns out that for $g(x,\cdot)=g_{a(x)}(\cdot)$ and $g(x,\cdot)=g_{b(x)}(\cdot)$ condition
\eqref{eq1.4} is valid if $\mu(r)=\ln\dfrac{1}{r}$.
We also note that for $\widetilde{g}_{a(x)}(\cdot)$ condition \eqref{eq1.4} is valid if
$\mu(r)=\ln^{\alpha}\dfrac{1}{r}$ and $0\leqslant \alpha+\dfrac{\beta}{q-p}\leqslant1$.
Thus, for double-phase elliptic equations, it was possible to substantially refine the condition
\eqref{eq1.3}.

The aim of this paper is to improve condition \eqref{eq1.3} for parabolic equations \eqref{eq1.1}
under non-logarithmic condition (${\rm g}_{2}$). Before formulating the main results, let us
recall the definition of a bounded weak solution to Eq. \eqref{eq1.1}.
We will use the well-known notation for sets, function spaces and for their elements
(see \cite{DiBenedettoDegParEq, LadUr} for references).
\begin{definition}
{\rm
We say that $u$ is a bounded weak sub(super) solution of Eq. \eqref{eq1.1} if
$u\in C_{{\rm loc}}(0,T; L^{2}_{{\rm loc}}(\Omega))\cap
L^{q}_{{\rm loc}}(0,T; W^{1,q}_{{\rm loc}}(\Omega))\cap L^{\infty}(\Omega_{T})$, and for any
compact $E\subset \Omega$ and for every subinterval $[t_{1},t_{2}]\subset(0,T]$ the integral
identity
\begin{equation}\label{eq1.5}
\int\limits_{E}u\varphi\,dx\bigg|_{t_{1}}^{t_{2}}+\int\limits_{t_{1}}^{t_{2}}\int\limits_{E}
\left\{ -u\varphi_{\tau}+ \mathbf{A}(x,\tau, \nabla u) \nabla\varphi \right\}dxd\tau
\leqslant (\geqslant)\, 0
\end{equation}
holds true for any testing function $\varphi\in W^{1,2}(0,T; L^{2}(E))
\cap L^{q}(0,T; W^{1,q}_{0}(\Omega))$, $\varphi\geqslant0$.
}
\end{definition}

It would be technically convenient to have a formulation of weak solution that involves $u_{t}$.
Let $\rho(x)\in C_{0}^{\infty}(\mathbb{R}^{n})$, $\rho(x)\geqslant0$, $\rho(x)\equiv 0$ for
$|x|>1$ and $\int_{\mathbb{R}^{n}}\rho(x)\,dx=1$, and set
$$
\rho_{h}(x):=h^{-n}\rho(x/h), \quad
u_{h}(x,t):= h^{-1}\int\limits_{t}^{t+h}\int\limits_{\mathbb{R}^{n}}
u(y,\tau)\rho_{h}(x-y)\,dyd\tau.
$$
Fix $t\in(0,T)$ and let $h>0$ be so small that $0<t<t+h<T$. In \eqref{eq1.5} take $t_{1}=t$,
$t_{2}=t+h$ and replace $\varphi$ by $\int_{\mathbb{R}^{n}}\varphi(y,t)\rho_{h}(x-y)\,dy$.
Dividing by $h$, since the testing function does not depend on $\tau$, we obtain
\begin{equation}\label{eq1.6}
\int\limits_{E\times\{t\}}\left\{ \frac{\partial u_{h}}{\partial t}\,\varphi+
[\mathbf{A}(x,t,\nabla u)]_{h}\nabla\varphi \right\}dx\leqslant(\geqslant)\,0,
\end{equation}
for all $t\in(0,T-h)$ and for all non-negative $\varphi\in W^{1,q}_{0}(E)$.

Similarly to that of \cite[Chapter~2]{DiBenedettoDegParEq} we can prove that if $u$ is a
weak sub(super) solution to \eqref{eq1.1} then the truncations $+(u-k)_{+}$, $-(u-k)_{-}$
for all $k\in \mathbb{R}$ are weak sub(super) solutions to \eqref{eq1.1} in the sense
\eqref{eq1.6} with $\mathbf{A}(x,t,\nabla u)$ replaced by $\mathbf{A}(x,t,\pm\nabla (u-k)_{\pm})$.

To formulate our results we also need some additional assumptions on the function
$g(x,t,{\rm v})$. Further we will show that these conditions arise naturally.
Fix $(x_{0},t_{0})\in \Omega_{T}$ and set
$$
\psi(x,t,{\rm v}):=\frac{g(x,t,{\rm v})}{{\rm v}}, \ \ (x,t)\in \Omega_{T},
\ \ {\rm v}>0.
$$
We assume that there exist $R_{0}>0$, $b_{0}$, $\delta\geqslant0$ such that the function
$\psi$ satisfies one of the following conditions:

\begin{itemize}
\item[($\Psi_{1}$)]
(''\textbf{degenerate}'' \textbf{case})
$\psi(x_{0},t_{0}, {\rm v})$ is non-decreasing for ${\rm v}>b_{0}R_{0}^{-\delta}$,
\end{itemize}
or
\begin{itemize}
\item[($\Psi_{2}$)]
(''\textbf{singular}'' \textbf{case})
$\psi(x_{0},t_{0}, {\rm v})$ is non-increasing for ${\rm v}>b_{0}R_{0}^{-\delta}$.
\end{itemize}
Note that in the case $p\geqslant2$ or $q\leqslant2$ conditions ($\Psi_{1}$) or ($\Psi_{2}$) with
$b_{0}=\delta=0$, respectively, are consequences of condition (${\rm g}_{1}$).

\begin{remark}\label{rem1.5}
{\rm
%the function $\psi_{a(x_{0},t_{0})}({\rm v})={\rm v}^{\,p-2}+a(x_{0},t_{0}){\rm v}^{\,q-2}$,
%${\rm v}>0$, satisfies condition ($\Psi_{1}$) with $b_{0}=\delta=0$ if $p\geqslant2$.
%In the case $q\leqslant2$ the function $\psi_{a(x_{0},t_{0})}({\rm v})$ satisfies ($\Psi_{2}$)
%with $b_{0}=\delta=0$.
If $p<2<q$ and $a(x_{0},t_{0})=0$, then the function
$$
\psi_{a(x_{0},t_{0})}({\rm v})={\rm v}^{\,p-2}+a(x_{0},t_{0}){\rm v}^{\,q-2}, \quad
{\rm v}>0,
$$
satisfies ($\Psi_{2}$) with $b_{0}=\delta=0$. And if $a(x_{0},t_{0})>0$ and $p<2<q$, then we
choose $R_{0}$ from the condition
$AR_{0}^{\,q-p}\mu^{q-p}(R_{0})=\dfrac{1}{2}\,a(x_{0},t_{0})$, then since
$$
\psi'_{a(x_{0},t_{0})}({\rm v})=
{\rm v}^{\,p-3}(p-2+(q-2)a(x_{0},t_{0}){\rm v}^{\,q-p})\geqslant0
\quad \text{if} \ \
{\rm v}\geqslant \left( \frac{2-p}{(q-2)\,a(x_{0},t_{0})} \right)^{\frac{1}{q-p}},
$$
which implies ($\Psi_{1}$) with $\delta=1$ and
$b_{0}=\left( \dfrac{2-p}{2A(q-2)} \right)^{\frac{1}{q-p}}$.

Similarly, the function
$\widetilde{\psi}_{a(x_{0},t_{0})}({\rm v})={\rm v}^{\,p-2}
+a(x_{0},t_{0}){\rm v}^{\,\widetilde{q}-2}[1+\ln(1+{\rm v})]^{\beta}$,
$\beta>0$, ${\rm v}>0$, satisfies condition
%($\Psi_{1}$) with $b_{0}=\delta=0$
%if $p\geqslant2$.
%If $q<2$ the function $\widetilde{\psi}_{a(x_{0},t_{0})}({\rm v})$
%satisfies condition ($\Psi_{2}$) with $b_{0}=e^{\frac{\beta}{2-q}}-1$ and $\delta=0$.
%If $p<2<q$ and $a(x_{0},t_{0})=0$, then $\widetilde{\psi}_{a(x_{0},t_{0})}({\rm v})$
%satisfies
($\Psi_{2}$) with $b_{0}=\delta=0$ if $p<2<q=\widetilde{q}+\beta$ and $a(x_{0},t_{0})=0$. % and $a(x_{0},t_{0})=0$.
Moreover, if $p<2<\widetilde{q}$ and $a(x_{0},t_{0})>0$, then
$\widetilde{\psi}_{a(x_{0},t_{0})}({\rm v})$ satisfies condition ($\Psi_{1}$) with $\delta=1$
and $b_{0}=\left( \dfrac{2-p}{2A(\widetilde{q}-2)} \right)^{\frac{1}{\widetilde{q}-p}}$.
In addition, if $\widetilde{q}<2<\widetilde{q}+\beta$ and $a(x_{0},t_{0})>0$, then
$$
\widetilde{\psi}'_{a(x_{0},t_{0})}({\rm v})\leqslant a(x_{0},t_{0})
{\rm v}^{\widetilde{q}-p}\,[1+\ln(1+{\rm v})]^{\beta}
\left(\widetilde{q}-2+\frac{\beta}{1+\ln(1+{\rm v})}  \right)\leqslant0
\quad \text{if} \ \ {\rm v}\geqslant e^{\frac{\beta}{2-\widetilde{q}}}-1,
$$
which implies ($\Psi_{2}$) with $\delta=0$ and $b_{0}=e^{\frac{\beta}{2-\widetilde{q}}}-1$.
And finally, if $\widetilde{q}=2$ and $a(x_{0},t_{0})>0$,
then
%$$
%\begin{aligned}
%\widetilde{\psi}'_{a(x_{0},t_{0})}({\rm v})
%&=\frac{{\rm v}^{p-3}}{1+{\rm v}}
%\left( \beta a(x_{0},t_{0}) {\rm v}^{\widetilde{q}-p+1}[1+\ln(1+{\rm v})]^{\beta-1}-(2-p)(1+{\rm v}) \right)
%\\
%&\geqslant \frac{{\rm v}^{p-3}}{1+{\rm v}} \left( a(x_{0},t_{0}) {\rm v}^{\widetilde{q}-p+1}-2(2-p){\rm v} \right)
%\geqslant 0
%\end{aligned}
%$$
%for $\beta\geqslant1$ and
%${\rm v}\geqslant \left( \dfrac{2(2-p)}{a(x_{0},t_{0})} \right)^{\frac{1}{\widetilde{q}-p}}+1$.
%And,
$$
\begin{aligned}
\widetilde{\psi}'_{a(x_{0},t_{0})}({\rm v})&=
\frac{{\rm v}^{p-3}}{1+{\rm v}}
\left( \beta a(x_{0},t_{0})\,{\rm v}^{3-p}[1+\ln(1+{\rm v})]^{\beta-1}-(2-p)(1+{\rm v})  \right)
\\
&\geqslant\frac{{\rm v}^{p-2}}{1+{\rm v}}
\left( \beta a(x_{0},t_{0})\, {\rm v}^{2-p}[1+\ln(1+{\rm v})]^{-1}-2(2-p)\right)
\\
&\geqslant \frac{(2-p){\rm v}^{p-2}}{1+{\rm v}} \left( \frac{\beta}{4}\,
a(x_{0},t_{0})\, {\rm v}^{\frac{2-p}{2}}-2\right)
\geqslant 0
\end{aligned}
$$
for
${\rm v}\geqslant \left( \dfrac{8}{\beta a(x_{0},t_{0})} \right)^{\frac{2}{2-p}}+1$.
Choosing $R_{0}$ from the condition
$$
AR_{0}^{\,2-p}\mu^{\,2-p}(R_{0})=\frac{1}{2}\,a(x_{0},t_{0}),
$$
we arrive at condition ($\Psi_{1}$) with $\delta=2$ and
$b_{0}=1+\left(\dfrac{4}{\beta A} \right)^{\frac{2}{2-p}}$.
}
\end{remark}
\begin{remark}\label{rem1.6}
{\rm
%the function $\psi_{b(x_{0},t_{0})}({\rm v})={\rm v}^{\,p-2}
%\left[ 1+\ln\big(1+b(x_{0},t_{0}){\rm v}\big) \right]$, ${\rm v}>0$,
%satisfies ($\Psi_{1}$) with $b_{0}=\delta=0$ if $p\geqslant2$.
If $p<2$ and $b(x_{0},t_{0})=0$, then the function
$$
\psi_{b(x_{0},t_{0})}({\rm v})={\rm v}^{\,p-2}
\left[ 1+\ln\big(1+b(x_{0},t_{0}){\rm v}\big) \right],
\ \ {\rm v}>0,
$$
satisfies condition ($\Psi_{2}$) with $b_{0}=\delta=0$. To chek condition
($\Psi_{2}$) in the case $p<2$ and $b(x_{0},t_{0})>0$ we note that
$$
\begin{aligned}
\psi'_{b(x_{0},t_{0})}({\rm v})
&=(p-2){\rm v}^{\,p-3}
\left[ 1+\ln\big(1+b(x_{0},t_{0}){\rm v}\big) \right]
+{\rm v}^{\,p-2} \frac{b(x_{0},t_{0})}{1+b(x_{0},t_{0}){\rm v}}
\\
&={\rm v}^{\,p-3}\left[ 1+\ln\big(1+b(x_{0},t_{0}){\rm v}\big) \right]
\left( p-2+\frac{b(x_{0},t_{0}){\rm v}}{1+b(x_{0},t_{0}){\rm v}}\,
\frac{1}{1+\ln\big(1+b(x_{0},t_{0}){\rm v}\big)} \right)
\\
&\leqslant {\rm v}^{\,p-3}\left[ 1+\ln\big(1+b(x_{0},t_{0}){\rm v}\big) \right]
\left( p-2+
\frac{1}{1+\ln\big(1+b(x_{0},t_{0}){\rm v}\big)} \right)
\\
&\leqslant 0, \quad \text{if} \ \
{\rm v}\geqslant \frac{e^{\frac{1}{2-p}}-1}{b(x_{0},t_{0})}.
\end{aligned}
$$
Choosing $R_{0}$ from the condition $BR_{0}\mu(R_{0})=\dfrac{1}{2}\,b(x_{0},t_{0})$,
we arrive at ($\Psi_{2}$) with $\delta=1$ and $b_{0}=(2B)^{-1}(e^{\frac{1}{2-p}}-1)$.
}
\end{remark}
\begin{remark}\label{rem1.7}
{\rm
We note that in this paper, conditions ($\Psi_{1}$) and ($\Psi_{2}$) on the function $\psi$
are weaker and more natural in comparison with conditions (${\rm g}_{12}$) and (${\rm g}_{22}$)
of paper \cite{SkrVoitNA20}. All results of \cite{SkrVoitNA20} can be generalized to cases
($\Psi_{1}$) and ($\Psi_{2}$), additionally using logarithmic estimates (see below Section
\ref{Sect2}, inequality \eqref{eq2.4} and Section \ref{Sect3},
Lemmas \ref{lem3.3} and \ref{lem3.4}). We leave out the detailes for which we refer the reader
to \cite{SkrVoitNA20}. In addition, we note that conditions (${\rm g}_{12}$) and (${\rm g}_{22}$)
from \cite{SkrVoitNA20} for double-phase parabolic equations
(i.e. if $g(x,t,\cdot)=g_{a(x,t)}(\cdot)$) do not cover the case $a(x_{0},t_{0})=0$ and $p<2<q$.
In this paper, this case is studied.
For other well-known cases when $1<p\leqslant q\leqslant 2$ or
$2\leqslant p \leqslant q < +\infty$,
and coefficients $\mathbf{A}$
in \eqref{eq1.1} are independent of $x$ and $t$, we refer the reader to the papers of Hwang and
Lieberman \cite{HwangLieberman287, HwangLieberman288}.
}
\end{remark}

%Note that the constants $b_{0}$ and $\delta$ cab be equal to zero, in the proof we keep
%an explicit track of the dependence of the various constants on $b_{0}$ and $\delta$.

Our main result of this paper reads as follows:
\begin{theorem}\label{th1.1}
Let $u$ be a bounded weak solution to Eq. \eqref{eq1.1} in $\Omega_{T}$. Fix
$(x_{0},t_{0})\in \Omega_{T}$ and let conditions $({\rm g}_{1})$, $({\rm g}_{2})$,
$(\Psi_{1})$ or $(\Psi_{2})$ be fulfilled in some cylinder
$Q_{R_{0}, R_{0}}(x_{0},t_{0})\subset\Omega_{T}$. Assume also that
\begin{equation}\label{eq1.7}
\int_{0}\lambda(r)\,\frac{dr}{r}=+\infty,
\end{equation}
then $u$ is continuous at $(x_{0},t_{0})$.
Moreover, in the case $p\geqslant2$ or $q\leqslant2$, $u\in C_{{\rm loc}}(\Omega_{T})$.
\end{theorem}
\begin{remark}\label{rem1.8}
{\rm
We note that in the case when $g(x,t,\cdot)=g_{a(x,t)}(\cdot)$ or $g(x,t,\cdot)=g_{b(x,t)}(\cdot)$
condition \eqref{eq1.7} can be rewritten in the form
$$
\int_{0} \frac{1}{\mu(r)}\,\frac{dr}{r}=+\infty.
$$
The function $\mu(r)=\ln\dfrac{1}{r}$ satisfies the above condition.

We also note that in the case $g(x,t,\cdot)=\widetilde{g}_{a(x,t)}(\cdot)$
condition \eqref{eq1.7} can be rewritten as
$$
\int_{0} \frac{\ln^{-\frac{\beta}{q-p}}\frac{1}{r}}{\mu(r)}\,
\frac{dr}{r}=+\infty.
$$
The function $\mu(r)=\ln^{\alpha}\dfrac{1}{r}$,
$0\leqslant\alpha+\dfrac{\beta}{q-p}\leqslant1$,
satisfies the above condition.
}
\end{remark}
\begin{remark}\label{rem1.9}
{\rm
Note that for the evolution $p(x,t)$-Laplace equation in the case when $p(x_{0},t_{0})=2$, or for
the case $g(x,t,{\rm v})={\rm v}^{\,p-1}\big(1+b(x,t)\ln(1+{\rm v})\big)$ and $p=2$, the result of
Theorem \ref{th1.1} seems new even for $\lambda(r)=1$.
}
\end{remark}

Now let's say a few words about the structure of the rest of the article and the approaches used.
To prove the continuity of solutions
to Eq. \eqref{eq1.1} (Theorem \ref{th1.1}), we develop Di\,Benedetto's innovative intrinsic scaling method.
We use parabolic $\mathcal{B}_{1,g}$ classes of De\,Giorgi-Ladyzhenskaya-Ural'tseva (see Section \ref{Sect2}),
which were essentially defined by the authors in \cite{SkrVoitNA20}.
In Section \ref{Sect2}, we assert the point-wise continuity of functions
belonging to these classes under conditions
$({\rm g}_{1})$, $({\rm g}_{2})$, \eqref{eq1.7}, $(\Psi_{1})$ or $(\Psi_{2})$
(see Theorem \ref{th2.1}).
We also show that solutions of Eq. \eqref{eq1.1} belong to $\mathcal{B}_{1,g}$ classes.
This fact and Theorem \ref{th2.1} imply the validity of Theorem \ref{th1.1}.
Section \ref{Sect3} contains auxiliary material
(De\,Giorgi-Poincar\'{e} inequality, De\,Giorgi-type lemmas, etc.) needed to prove Theorem \ref{th2.1}.
Finally, the proof of Theorem \ref{th2.1} is given in Sections \ref{Sect4} and \ref{Sect5}
for the degenerate and singular cases (the cases ($\Psi_{1}$) and ($\Psi_{2}$)), respectively.
In addition, note that for the double-phase equations
(i.e. if $g(x,t,\cdot)=g_{a(x,t)}(\cdot)$) we cover the case $a(x_{0},t_{0})=0$ and
$p<2<q$, which has not been studied previously.

%Now few words concerning the approach taken in this paper. To prove the continuity of solutions
%to Eq. \eqref{eq1.1} we use Di\,Benedetto's innovative intrinsic scaling method. We use parabolic
%$\mathcal{B}_{1,g}$ classes of De\,Giorgi-Ladyzhenskaya-Ural'tseva, which were essentially defined
%by the authors in \cite{SkrVoitNA20}. We show that under conditions
%$({\rm g}_{1})$, $({\rm g}_{2})$, \eqref{eq1.7}, $(\Psi_{1})$ or $(\Psi_{2})$, the functions
%belonging to these classes are continuous. We also show that solutions of Eq. \eqref{eq1.1}
%belong to these classes. In addition, note that for double-phase equations
%(i.e. if $g(x,t,\cdot)=g_{a(x,t)}(\cdot)$) we cover the case $a(x_{0},t_{0})=0$ and
%$p<2<q$, which has not been studied previously.
%Our paper is organized as follows.

%%%%%%%%%%%%%%%%%%%%%%%%%%%%%%%%%%%%%%%%%%%%%%%%%%%%%%%%%%%%%%%%%%%%%%%%%%%%%%%%%%%%%%%%%%%%%%%%%%%%%%%%%%%%%%%%%%%%%%%%%%%%%%%%%
%%%%%%%%%%%%%%%%%%%%%%%%%%%%%%%%%%%%%%%%%%%%%%%%%%%%%%%%%%%%%%%%%%%%%%%%%%%%%%%%%%%%%%%%%%%%%%%%%%%%%%%%%%%%%%%%%%%%%%%%%%%%%%%%%
%%%%%%%%%%%%%%%%%%%%%%%%%%%%%%%%%%%%%%%%%%%%%%%%%%%%%%%%%%%%%%%%%%%%%%%%%%%%%%%%%%%%%%%%%%%%%%%%%%%%%%%%%%%%%%%%%%%%%%%%%%%%%%%%%

\section{Parabolic $\mathcal{B}_{1,g}(\Omega_{T})$ classes}\label{Sect2}

As it was already mentioned, parabolic $\mathcal{B}_{1,g}(\Omega_{T})$ classes were practically
defined in \cite{SkrVoitNA20}. In this paper we give their generalization, this is due to the
fact that we consider the cases $p=2$ or $q=2$.

\begin{definition}
{\rm
We say that a measurable function $u:\Omega_{T}\rightarrow\mathbb{R}$
belongs to the parabolic class
$\mathcal{B}_{1,g}(\Omega_{T})$, if
$u\in C_{{\rm loc}}(0,T;L^{2}_{{\rm loc}}(\Omega))\cap
L^{1}_{{\rm loc}}(0,T; W^{1,1}_{{\rm loc}}(\Omega))\cap L^{\infty}(\Omega_{T})$,
$\esssup\limits_{\Omega_{T}}|u|\leqslant M$ and there exist numbers $1<p<q$,
$c_{2}\geqslant q$, $c_{3}>0$ such that for any cylinder
$Q_{8r,8\theta}(\overline{x}, \overline{t})\subset Q_{8r,8r}(\overline{x}, \overline{t})
\subset \Omega_{T}$, any $k$, $l\in \mathbb{R}$, $k<l$, $|k|$, $|l|<M$, any
$\varepsilon\in (0,1]$, any $\sigma\in(0,1)$, for any
$\zeta(x)\in C_{0}^{\infty}(B_{r}(\overline{x}))$, $0\leqslant\zeta(x)\leqslant1$,
$\zeta(x)=1$ in $B_{r(1-\sigma)}(\overline{x})$,
$|\nabla \zeta|\leqslant(\sigma r)^{-1}$, and for any
$\chi(t)\in C_{0}^{\infty}(\mathbb{R}_{+})$,
$0\leqslant\chi(t)\leqslant1$, the following inequalities hold:
\begin{equation}\label{eq2.1}
\begin{aligned}
&\iint\limits_{A^{+}_{k,r,\theta}\setminus A^{+}_{l,r,\theta}}
g\left(x,t, \frac{M_{+}(k,r,\theta)}{r} \right)|\nabla u|\,\zeta^{\,c_{2}}\chi\, dxdt
\\
&\leqslant
c_{3}\,\frac{M_{+}(k,r,\theta)}{\varepsilon r}
\iint\limits_{A^{+}_{k,r,\theta}\setminus A^{+}_{l,r,\theta}}
g\left(x,t, \frac{M_{+}(k,r,\theta)}{r} \right)dxdt
\\
&+c_{3}\sigma^{-q}\varepsilon^{\,p-1}
\bigg\{ \int\limits_{B_{r}(\overline{x})\times\{\overline{t}-\theta\}}
(u-k)_{+}^{2}\,\zeta^{\,c_{2}}\chi(\overline{t}-\theta)\, dx+
\iint\limits_{A^{+}_{k,r,\theta}}(u-k)_{+}^{2}\,\zeta^{\,c_{2}}|\,\chi_{t}|\,dxdt
\\
&\hskip 25mm+ \iint\limits_{A^{+}_{k,r,\theta}}
g\left(x,t, \frac{(u-k)_{+}}{r}  \right) \frac{(u-k)_{+}}{r}\,\zeta^{\,c_{2}-q}\chi\,dxdt \bigg\},
\end{aligned}
\end{equation}

\begin{equation}\label{eq2.2}
\begin{aligned}
&\iint\limits_{A^{-}_{l,r,\theta}\setminus A^{-}_{k,r,\theta}}
g\left(x,t, \frac{M_{-}(l,r,\theta)}{r} \right)|\nabla u|\,\zeta^{\,c_{2}}\chi\, dxdt
\\
&\leqslant
c_{3}\,\frac{M_{-}(l,r,\theta)}{\varepsilon r}
\iint\limits_{A^{-}_{l,r,\theta}\setminus A^{-}_{k,r,\theta}}
g\left(x,t, \frac{M_{-}(l,r,\theta)}{r} \right)dxdt
\\
&+c_{3}\sigma^{-q}\varepsilon^{\,p-1}
\bigg\{ \int\limits_{B_{r}(\overline{x})\times\{\overline{t}-\theta\}}
(u-l)_{-}^{2}\,\zeta^{\,c_{2}}\chi(\overline{t}-\theta)\, dx+
\iint\limits_{A^{-}_{l,r,\theta}}(u-l)_{-}^{2}\,\zeta^{\,c_{2}}|\,\chi_{t}|\,dxdt
\\
&\hskip 25mm+ \iint\limits_{A^{-}_{l,r,\theta}}
g\left(x,t, \frac{(u-l)_{-}}{r}  \right) \frac{(u-l)_{-}}{r}\,\zeta^{\,c_{2}-q}\chi\,dxdt \bigg\},
\end{aligned}
\end{equation}

\begin{equation}\label{eq2.3}
\begin{aligned}
&\int\limits_{B_{r}(\overline{x})\times\{t\}}
(u-k)_{\pm}^{2}\,\zeta^{\,c_{2}}\chi\,dx
\leqslant
\int\limits_{B_{r}(\overline{x})\times\{\overline{t}-\theta\}}
(u-k)_{\pm}^{2}\,\zeta^{\,c_{2}}\chi(\overline{t}-\theta)\,dx
\\
&+c_{3}\sigma^{-1}
\bigg\{\iint\limits_{A^{\pm}_{k,r,\theta}}
(u-k)_{\pm}^{2}\,\zeta^{\,c_{2}}|\chi_{t}|\,dxdt+\iint\limits_{A^{\pm}_{k,r,\theta}}
g\left(x,t, \frac{(u-k)_{\pm}}{r}  \right)
\frac{(u-k)_{\pm}}{r}\,\zeta^{\,c_{2}-q}\chi\,dxdt \bigg\},
\\
&\hskip 112mm \text{for all} \ t\in(\overline{t}-\theta, \overline{t}),
\end{aligned}
\end{equation}

\begin{equation}\label{eq2.4}
\begin{aligned}
&\int\limits_{B_{r}(\overline{x})\times\{t\}}
\ln_{+}^{2}\frac{M_{\pm}(k,r,\theta)}{M_{\pm}(k,r,\theta)-(u-k)_{\pm}+a}\,
\zeta^{\,c_{2}}\,dx
\\
&\leqslant
\int\limits_{B_{r}(\overline{x})\times\{\overline{t}-\theta\}}
\ln_{+}^{2}\frac{M_{\pm}(k,r,\theta)}{M_{\pm}(k,r,\theta)-(u-k)_{\pm}+a}\,
\zeta^{\,c_{2}}\,dx
\\
&+c_{3}\sigma^{-q}r^{-2}\ln\frac{M_{\pm}(k,r,\theta)}{a}
\iint\limits_{A^{\pm}_{k,r,\theta}}
\psi\left( x,t, \frac{M_{\pm}(k,r,\theta)-(u-k)_{\pm}+a}{r} \right)
\,\zeta^{\,c_{2}-q}\,dxdt,
\\
& \hskip 54,5mm \text{for all} \ t\in(\overline{t}-\theta, \overline{t})
\ \text{and for} \ 0<a< M_{\pm}(k,r,\theta).
\end{aligned}
\end{equation}
Here $(u-k)_{\pm}:=\max\{\pm(u-k), 0\}$,
$M_{\pm}(k,r,\theta):=\esssup\limits_{Q_{r,\theta}(\overline{x},\overline{t})} (u-k)_{\pm}$
and
\newline
$A^{\pm}_{k,r,\theta}:=Q_{r,\theta}(\overline{x},\overline{t})\cap \{(u-k)_{\pm}>0\}$.
}
\end{definition}

As already noted, we will distinguish two cases: ''degenerate'' and ''singular''.
Fix $(x_{0},t_{0})\in \Omega_{T}$,
in the ''singular'' case, i.e. if condition $(\Psi_{2})$ is true, we additionally assume
that for all $t\in(t_{0}-\theta, t_{0})$ and for any $k>0$, $\varepsilon$,
$\varepsilon_{1}\in(0,1)$ there holds:
\begin{equation}\label{eq2.5}
\begin{aligned}
&D^{-}\int\limits_{B_{r}(x_{0})\times\{t\}}
\Phi_{k}(x_{0},t_{0},v_{\pm})\,\frac{t-t_{0}+\theta}{\theta}\,\zeta^{\,c_{2}}\,dx
\\
&+\frac{r\varepsilon_{1}^{1-p}}{c_{3}}
\int\limits_{B_{r}(x_{0})\times\{t\}}
\left| \nabla \ln\frac{(1+\varepsilon)k}{w_{k,\varepsilon}} \right|
\frac{g(x,t,w_{k,\varepsilon}/r)}{g(x_{0},t_{0},w_{k,\varepsilon}/r)}\,
\frac{t-t_{0}+\theta}{\theta}\,\zeta^{\,c_{2}}\,dx
\\
&\leqslant
\frac{c_{3}}{\theta} \int\limits_{B_{r}(x_{0})\times\{t\}}
\Phi_{k}(x_{0},t_{0},v_{\pm})\,\zeta^{\,c_{2}}\,dx
\\
&+
c_{3}(1+\varepsilon_{1}^{-p})\sigma^{-q}\int\limits_{B_{r}(x_{0})\times\{t\}}
\frac{g(x,t,w_{k,\varepsilon}/r)}{g(x_{0},t_{0},w_{k,\varepsilon}/r)}\,
\zeta^{\,c_{2}-q}\,dx,
\end{aligned}
\end{equation}
where $v_{+}:=\mu_{+}-u$, $v_{-}:=u-\mu_{-}$,
$\mu_{+}\geqslant \esssup\limits_{Q_{r,\theta}(x_{0},t_{0})}u$,
$\mu_{-}\leqslant \essinf\limits_{Q_{r,\theta}(x_{0},t_{0})}u$,
$w_{k,\varepsilon}=k(1+\varepsilon)-(v_{\pm}-k)_{-}$,
$$
\Phi_{k}(x,t,v_{\pm}):=\int\limits_{0}^{(v_{\pm}-k)_{-}}
\frac{(1+\varepsilon)k-s}{\mathcal{G}\left(x,t, \frac{(1+\varepsilon)k-s}{r} \right)}\,ds,
\quad
\mathcal{G}(x,t,{\rm v}):=\int\limits_{0}^{{\rm v}}g(x,t,s)\,ds,
$$
and the notation $D^{-}$ is used to denote the derivative
$$
D^{-}f(t):=\limsup\limits_{h\rightarrow0}
\frac{f(t)-f(t-h)}{h}.
$$

The parameters $n$, $p$, $q$, $K_{1}$, $K_{2}$, $K_{3}$, $M$, $c_{1}(2M)$, $c_{2}$, $c_{3}$, are the data,
and we say that a generic constant $\gamma$ depends only upon the data, if it can be
quantitatively determined a priory only in termms of the indicated parameters.
Note that the constants $b_{0}$ and $\delta$ can be equal to zero, in the proof we keep an explicit track
of the dependence of the various constants on $b_{0}$ and $\delta$.

Our main result of this Section reads as follows.

\begin{theorem}\label{th2.1}
Let $u\in \mathcal{B}_{1,g}(\Omega_{T})$, fix $(x_{0},t_{0})\in \Omega_{T}$ such that
$Q_{8R_{0},8R_{0}}(x_{0},t_{0})\subset\Omega_{T}$ and let hypotheses
$({\rm g}_{1})$, $({\rm g}_{2})$, \eqref{eq1.7}, $(\Psi_{1})$ or $(\Psi_{2})$ be fulfilled
in $Q_{R_{0},R_{0}}(x_{0},t_{0})$. Then $u$ is continuous at $(x_{0},t_{0})$.
Moreover, if $p\geqslant2$ or $q\leqslant2$, then $u\in C_{{\rm loc}}(\Omega_{T})$.
\end{theorem}

We note that the solutions of Eq. \eqref{eq1.1} belong to the correspondig
$\mathcal{B}_{1,g}(\Omega_{T})$ classes. The proof of inequalities
\eqref{eq2.1}--\eqref{eq2.3} is completely similar to the proof of inequalities
(3.1)--(3.3) from \cite[Sect.~4]{SkrVoitNA20}. Let us show inequalities
\eqref{eq2.4} and \eqref{eq2.5}. First note a simple analogues of Young inequality:
\begin{equation}\label{eq2.6}
g(x,t,a)b\leqslant \varepsilon g(x,t,a)a+g(x,t,b/\varepsilon)b,
\quad (x,t)\in \Omega_{T}, \ \ a, b, \varepsilon>0.
\end{equation}

Testing identity \eqref{eq1.6} by
$\left( \ln^{2}_{+}\dfrac{M_{\pm}(k,r,\theta)}{M_{\pm}(k,r,\theta)-(u_{h}-k)_{\pm}+a} \right)'_{u}
\zeta^{\,c_{2}}(x)$, where $\zeta$ is the same as in \eqref{eq2.4}, integrating over
$(\overline{t}-\theta,t)$, $t\in(\overline{t}-\theta,\overline{t})$, then integrating by parts in
the term containing $\dfrac{\partial u_{h}}{\partial t}$, letting $h\rightarrow0$, we obtain
$$
\begin{aligned}
&\int\limits_{B_{r}(\overline{x})\times\{t\}}
\ln^{2}_{+}\dfrac{M_{\pm}(k,r,\theta)}{M_{\pm}(k,r,\theta)-(u-k)_{\pm}+a}\,
\zeta^{\,c_{2}}\,dx
\\
&+\frac{1}{\gamma}\int\limits_{\overline{t}-\theta}^{t}\int\limits_{B_{r}(\overline{x})}
\left( 1+\ln_{+}\dfrac{M_{\pm}(k,r,\theta)}{M_{\pm}(k,r,\theta)-(u-k)_{\pm}+a} \right)
\frac{G(x,t,|\nabla(u-k)_{\pm}|)\,\zeta^{\,c_{2}}}
{(M_{\pm}(k,r,\theta)-(u-k)_{\pm}+a)^{2}}\,dxdt
\\
&\leqslant \int\limits_{B_{r}(\overline{x})\times\{\overline{t}-\theta\}}
\ln^{2}_{+}\dfrac{M_{\pm}(k,r,\theta)}{M_{\pm}(k,r,\theta)-(u-k)_{\pm}+a}\,
\zeta^{\,c_{2}}\,dx
\\
&+\gamma\sigma^{-2}
\int\limits_{\overline{t}-\theta}^{t}\int\limits_{B_{r}(\overline{x})}
\frac{g(x,t,|\nabla(u-k)_{\pm}|)\,\zeta^{\,c_{2}-1}}{M_{\pm}(k,r,\theta)-(u-k)_{\pm}+a}\,
\ln_{+}\dfrac{M_{\pm}(k,r,\theta)}{M_{\pm}(k,r,\theta)-(u-k)_{\pm}+a}\,dxdt.
\end{aligned}
$$
From this, using \eqref{eq2.6}, we arrive at the required \eqref{eq2.4}.

To prove \eqref{eq2.5} we test \eqref{eq1.5} by
$$
\varphi=\dfrac{w_{k,\varepsilon}\,\zeta^{\,c_{2}}}
{\mathcal{G}(x_{0},t_{0}, w_{k,\varepsilon}/r)}\,
\dfrac{t-t_{0}+\theta}{\theta},
$$
assuming that
$u_{t}\in C(t_{0}-\theta, t_{0}; L^{1}(B_{r}(x_{0})))$,
this condition can be removed similarly to that of
\cite[Chap.~4, pp.\,101--102]{DiBenedettoDegParEq}, We obtain for all
$t\in (t_{0}-\theta,t_{0})$
$$
\begin{aligned}
&\frac{\partial}{\partial t}\int\limits_{B_{r}(x_{0})\times\{t\}}
\Phi_{k}(x_{0},t_{0},v_{\pm})\,\frac{t-t_{0}+\theta}{\theta}\,
\zeta^{\,c_{2}}\,dx
\\
&+
\int\limits_{B_{r}(x_{0})\times\{t\}}
\frac{G(x,t,|\nabla(v_{\pm}-k)_{-}|)}{\mathcal{G}(x_{0},t_{0},w_{k,\varepsilon}/r)}
\left( \frac{G(x_{0},t_{0},w_{k,\varepsilon}/r)}
{\mathcal{G}(x_{0},t_{0},w_{k,\varepsilon}/r)}-1 \right)
\frac{t-t_{0}+\theta}{\theta}\,\zeta^{\,c_{2}}\,dx
\\
&\leqslant \frac{\gamma}{\theta}\int\limits_{B_{r}(x_{0})\times\{t\}}
\Phi_{k}(x_{0},t_{0},v_{\pm})\,\zeta^{\,c_{2}}\,dx+\gamma
\int\limits_{B_{r}(x_{0})\times\{t\}}
\frac{g(x,t,|\nabla(v_{\pm}-k)_{-}|)}{\mathcal{G}(x_{0},t_{0},w_{k,\varepsilon}/r)}
\,\frac{t-t_{0}+\theta}{\theta}\,\zeta^{\,c_{2}-1}\,dx.
\end{aligned}
$$
Condition $({\rm g}_{1})$ implies that
$q^{-1}G(x,t,{\rm w})\leqslant \mathcal{G}(x,t,{\rm w})\leqslant p^{-1}G(x,t,{\rm w})$
for $(x,t)\in\Omega_{T}$ and ${\rm w}>0$, so from the previous by $({\rm g}_{1})$ and
\eqref{eq2.6} we have
\begin{equation}\label{eq2.7}
\begin{aligned}
&\frac{\partial}{\partial t}\int\limits_{B_{r}(x_{0})\times\{t\}}
\Phi_{k}(x_{0},t_{0},v_{\pm})\,\frac{t-t_{0}+\theta}{\theta}\,
\zeta^{\,c_{2}}\,dx
\\
&+(p-1)\int\limits_{B_{r}(x_{0})\times\{t\}}
\frac{G(x,t,|\nabla(v_{\pm}-k)_{-}|)}{\mathcal{G}(x_{0},t_{0},w_{k,\varepsilon}/r)}
\,\frac{t-t_{0}+\theta}{\theta}\,\zeta^{\,c_{2}}\,dx
\\
&\leqslant \frac{\gamma}{\theta}\int\limits_{B_{r}(x_{0})\times\{t\}}
\Phi_{k}(x_{0},t_{0},v_{\pm})\,\zeta^{\,c_{2}}\,dx+
\frac{\gamma}{\sigma^{q}}\int\limits_{B_{r}(x_{0})\times\{t\}}
\frac{g(x,t,w_{k,\varepsilon}/r)}{g(x_{0},t_{0},w_{k,\varepsilon}/r)}\,
\zeta^{\,c_{2}-q}\,dx.
\end{aligned}
\end{equation}

Let us estimate the second term on the left-hand side of \eqref{eq2.7}.
Condition $({\rm g}_{1})$ and inequality \eqref{eq2.6} imply
$$
\begin{aligned}
&\frac{r}{\varepsilon_{1}^{\,p-1}}
\int\limits_{B_{r}(x_{0})\times\{t\}}
\left| \nabla\ln\frac{(1+\varepsilon)k}{w_{k,\varepsilon}} \right|
\frac{g(x,t,w_{k,\varepsilon}/r)}{g(x_{0},t_{0},w_{k,\varepsilon}/r)}\,
\frac{t-t_{0}+\theta}{\theta}\,\zeta^{\,c_{2}}\,dx
\\
&\leqslant \frac{\gamma}{\varepsilon_{1}^{\,p-1}}\int\limits_{B_{r}(x_{0})\times\{t\}}
|\nabla w_{k,\varepsilon}|\,\frac{g(x,t,w_{k,\varepsilon}/r)}
{\mathcal{G}(x_{0},t_{0},w_{k,\varepsilon}/r)}\,
\frac{t-t_{0}+\theta}{\theta}\,\zeta^{\,c_{2}}\,dx
\\
&\leqslant \frac{\gamma}{\varepsilon_{1}^{\,p}}\int\limits_{B_{r}(x_{0})\times\{t\}}
\frac{G(x,t,w_{k,\varepsilon}/r)}{\mathcal{G}(x_{0},t_{0},w_{k,\varepsilon}/r)}\,
\frac{t-t_{0}+\theta}{\theta}\,\zeta^{\,c_{2}}\,dx
\\
&+
\frac{\gamma}{\varepsilon_{1}^{\,p}}\int\limits_{B_{r}(x_{0})\times\{t\}}
\frac{G(x,t,\varepsilon_{1}
|\nabla w_{k,\varepsilon}|)}{\mathcal{G}(x_{0},t_{0},w_{k,\varepsilon}/r)}\,
\frac{t-t_{0}+\theta}{\theta}\,\zeta^{\,c_{2}}\,dx
\\
&\leqslant \frac{\gamma}{\varepsilon_{1}^{\,p}}\int\limits_{B_{r}(x_{0})\times\{t\}}
\frac{g(x,t,w_{k,\varepsilon}/r)}{g(x_{0},t_{0},w_{k,\varepsilon}/r)}\,
\frac{t-t_{0}+\theta}{\theta}\,\zeta^{\,c_{2}}\,dx
\\
&+\gamma\int\limits_{B_{r}(x_{0})\times\{t\}}
\frac{G(x,t,|\nabla w_{k,\varepsilon}|)}{\mathcal{G}(x_{0},t_{0},w_{k,\varepsilon}/r)}\,
\frac{t-t_{0}+\theta}{\theta}\,\zeta^{\,c_{2}}\,dx.
\end{aligned}
$$
From this and \eqref{eq2.7} we arrive at the required \eqref{eq2.5}.

Thus, taking into account the previous arguments, we obtain Theorem \ref{th1.1} as a
consequence of Theorem \ref{th2.1}.

%%%%%%%%%%%%%%%%%%%%%%%%%%%%%%%%%%%%%%%%%%%%%%%%%%%%%%%%%%%%%%%%%%%%%%%%%%%%%%%%%%%%%%%%%%%%%%%%%%%%%%%%%%%%%%%%%%%%%%%%%%%%%%%%%
%%%%%%%%%%%%%%%%%%%%%%%%%%%%%%%%%%%%%%%%%%%%%%%%%%%%%%%%%%%%%%%%%%%%%%%%%%%%%%%%%%%%%%%%%%%%%%%%%%%%%%%%%%%%%%%%%%%%%%%%%%%%%%%%%
%%%%%%%%%%%%%%%%%%%%%%%%%%%%%%%%%%%%%%%%%%%%%%%%%%%%%%%%%%%%%%%%%%%%%%%%%%%%%%%%%%%%%%%%%%%%%%%%%%%%%%%%%%%%%%%%%%%%%%%%%%%%%%%%%

\section{Auxiliary material and integral estimates}\label{Sect3}

\subsection{Auxiliary propositions}\label{subSect3.1}

For every Lebesgue measurable set $E\subset \mathbb{R}^{n}$,
we denote by $|E|$ the $n$-dimensional
Lebesgue measure of $E$ (or the $n+1$-dimensional measure, if $E\subset \mathbb{R}^{n+1}$).
The following two lemmas will be used in the sequel.
The first one is the well-known De\,Giorgi-Poincar\'{e} lemma (see \cite[Chap.\,II,\,Lemma\,3.9]{LadUr}).
\begin{lemma}\label{lem3.1}
Let $u\in W^{1,1}(B_{\rho}(x_{0}))$.
Then, for any $k,l\in \mathbb{R}$, $k<l$, the following
inequalities hold:
$$
(l-k)\,|A^{+}_{l,\rho}|
\leqslant
\frac{c\,\rho^{n}}{|B_{\rho}(x_{0})\setminus A^{+}_{k,\rho}|}
\int\limits_{A^{+}_{k,\rho}\setminus A^{+}_{l,\rho}} |\nabla u|\,dx,
$$
\begin{equation*}%\label{eq2.2}
(l-k)\,|A^{-}_{k,\rho}|
\leqslant
\frac{c\,\rho^{n}}{|B_{\rho}(x_{0})\setminus A^{-}_{l,\rho}|}
\int\limits_{A^{-}_{l,\rho}\setminus A^{-}_{k,\rho}} |\nabla u|\,dx,
\end{equation*}
where $A^{+}_{k,\rho}:= B_{\rho}(x_{0})\cap \{u>k\}$, $A^{-}_{k,\rho}:= B_{\rho}(x_{0})\cap \{u<k\}$
and $c$ is a positive constant depending only on $n$.
\end{lemma}

The following lemma can also be found in \cite[Chap.~II, Lemma~4.7]{LadUr}.
\begin{lemma}\label{lem3.2}
Let $y_{j}$, $j=0,1,2,\ldots$, be a sequence of non-negative numbers satisfying
$$
y_{j+1}\leqslant c\,b^{j}y_{j}^{1+\delta}, \quad j=0,1,2,\ldots,
$$
with some constants $\delta>0$ and $c,b>1$. Then
$$
y_{j}\leqslant c^{\frac{(1+\delta)^{j}-1}{\delta}}b^{\frac{(1+\delta)^{j}-1}{\delta^{2}}-\frac{j}{\delta}}
y_{0}^{(1+\delta)^{j}}, \quad j=0,1,2,\ldots.
$$
Particularly, if $y_{0}\leqslant \nu:=c^{-\frac{1}{\delta}}b^{-\frac{1}{\delta^{2}}}$, then
$$
y_{j}\leqslant \nu b^{-\frac{j}{\delta}} \quad \text{and} \quad \lim_{j\rightarrow\infty} y_{j}=0.
$$
\end{lemma}

%%%%%%%%%%%%%%%%%%%%%%%%%%%%%%%%%%%%%%%%%%%%%%%%%%%%%%%%%%%%%%%%%%%%%%%%%%%%%%%%%%%%%%%%%%%%%%%%%%%%%%%%%%%%%%%%%%%%%%
%%%%%%%%%%%%%%%%%%%%%%%%%%%%%%%%%%%%%%%%%%%%%%%%%%%%%%%%%%%%%%%%%%%%%%%%%%%%%%%%%%%%%%%%%%%%%%%%%%%%%%%%%%%%%%%%%%%%%%

\subsection{Local material and energy estimates}\label{subSect3.1}

We fix $(x_{0},t_{0})\in \Omega_{T}$, let $Q_{R_{0}, R_{0}}(x_{0},t_{0})\subset \Omega_{T}$ and
construct the cylinder
$$
Q_{r,\theta}(x_{0},\overline{t})\subset Q_{r,rK_{3}}(x_{0},t_{0})
\subset Q_{R_{0}, R_{0}}(x_{0},t_{0}) \ \ \text{for} \ r<R_{0}^{\overline{\delta}},
\ \overline{\delta}>1+\delta,
$$
where $\delta\geqslant0$ is the number that was defined in
assumptions $(\Psi_{1})$ and $(\Psi_{2})$. We denote by $\mu_{\pm}$ and $\omega$
non-negative numbers such that
$\mu_{+}\geqslant \esssup\limits_{Q_{r,\theta}(x_{0},\overline{t})}u$,
$\mu_{-}\leqslant \essinf\limits_{Q_{r,\theta}(x_{0},\overline{t})}u$,
$\omega=\mu_{+}-\mu_{-}$, and set $v_{+}:=\mu_{+}-u$, $v_{-}:=u-\mu_{-}$.

Further we need the following lemma.
\begin{lemma}\label{lem3.3}
Let $u\in \mathcal{B}_{1,g}(\Omega_{T})$ and let conditions
${\rm (g_{1})}$, ${\rm (g_{2})}$, $(\Psi_{1})$ or $(\Psi_{2})$ be
fulfilled in $Q_{R_{0}, R_{0}}(x_{0},t_{0})$. Assume also that
\begin{equation}\label{eq3.1}
\left|\left\{ x\in B_{r}(x_{0}):v_{\pm}(x,\overline{t}-\theta)
\leqslant \frac{\omega\,\lambda(r)}{2^{s_{0}}} \right\}  \right|
\leqslant (1-\beta_{0})\,|B_{r}(x_{0})|,
\end{equation}
for some $s_{0}\geqslant 1+\log M$ and $\beta_{0}\in(0,1)$. Then there exists
number $s_{1}>s_{0}$ depending only on the data, $s_{0}$ and $\beta_{0}$ such that
either
\begin{equation}\label{eq3.2}
\omega\leqslant 2^{s_{1}}(1+b_{0})\,\frac{r^{1-\frac{\delta}{\overline{\delta}}}}
{\lambda(r)},
\end{equation}
or
\begin{equation}\label{eq3.3}
\left|\left\{ x\in B_{r}(x_{0}):v_{\pm}(x,t)
\leqslant \frac{\omega\,\lambda(r)}{2^{s_{1}}} \right\}  \right|
\leqslant \left(1-\frac{\beta^{2}_{0}}{2}\right)\,|B_{r}(x_{0})|,
\end{equation}
for all $t\in(\overline{t}-\theta,\overline{t})$, provided that
$$
\theta\leqslant
\frac{r^{2}}{\psi\left( x_{0},t_{0}, \dfrac{\omega\,\lambda(r)}{2^{s_{0}}r} \right)}
\quad \text{in the case} \ (\Psi_{1}),
$$
or
$$
\theta\leqslant
\frac{r^{2}}{\psi\left( x_{0},t_{0}, \dfrac{\omega\,\lambda(r)}{2^{s_{1}}r} \right)}
\quad \text{in the case} \ (\Psi_{2}).
$$
\end{lemma}
\begin{proof}
We will prove inequality \eqref{eq3.3} for $v_{+}$, for $v_{-}$ the proof is completely similar.
For this we use inequality \eqref{eq2.4} with $k=\mu_{+}-2^{-s_{0}}\omega\,\lambda(r)$ and
$a=2^{-s_{1}}\omega\,\lambda(r)$ in the cylinder $Q_{r,\theta}(x_{0},\overline{t})$ and note that
$M_{+}(k,r,\theta)\leqslant 2^{-s_{0}}\omega\,\lambda(r)\leqslant \lambda(r)$. In addition, we
will assume that $M_{+}(k,r,\theta)\geqslant 2^{-s_{0}-1}\omega\,\lambda(r)$, because overwise,
inequality \eqref{eq3.3} is valid for $s_{1}=s_{0}+1$. In terms of $v_{+}$ inequality
\eqref{eq2.4} can be rewritten as
\begin{equation}\label{eq3.4}
\begin{aligned}
I_{1}&=\int\limits_{B_{r(1-\sigma)}(x_{0})\times\{t\}}
\ln_{+}^{2}\frac{M_{+}(k,r,\theta)}
{M_{+}(k,r,\theta)-(v_{+}-2^{-s_{0}}\omega\,\lambda(r))_{-}+2^{-s_{1}}\omega\,\lambda(r)}
\,dx
\\
&\leqslant \int\limits_{B_{r}(x_{0})\times\{\overline{t}-\theta\}}
\ln_{+}^{2}\frac{M_{+}(k,r,\theta)}
{M_{+}(k,r,\theta)-(v_{+}-2^{-s_{0}}\omega\,\lambda(r))_{-}+2^{-s_{1}}\omega\,\lambda(r)}
\,dx
\\
&+\frac{\gamma(s_{1}-s_{0})}{\sigma^{q}}\int\limits_{\overline{t}-\theta}^{t}
\int\limits_{B_{r}(x_{0})}
\psi\left( x,\tau,
\frac{M_{+}(k,r,\theta)-(v_{+}-2^{-s_{0}}\omega\,\lambda(r))_{-}+2^{-s_{1}}\omega\,\lambda(r)}
{r} \right) dxd\tau
\\
&=I_{2}+I_{3}.
\end{aligned}
\end{equation}
By \eqref{eq3.1} we have
\begin{equation}\label{eq3.5}
I_{2}\leqslant (s_{1}-s_{0})^{2}\ln^{2}2\,(1-\beta_{0})\,|B_{r}(x_{0})|.
\end{equation}
Using ${\rm (g_{2})}$, for $(x,\tau)\in B_{r}(x_{0})\times(\overline{t}-\theta, \overline{t})$,
we obtain
\begin{multline*}
g\left( x,\tau,
\frac{M_{+}(k,r,\theta)-(v_{+}-2^{-s_{0}}\omega\,\lambda(r))_{-}+2^{-s_{1}}\omega\,\lambda(r)}
{r} \right)
\\
\leqslant c_{2}\,
g\left( x_{0},t_{0},
\frac{M_{+}(k,r,\theta)-(v_{+}-2^{-s_{0}}\omega\,\lambda(r))_{-}+2^{-s_{1}}\omega\,\lambda(r)}
{r} \right).
\end{multline*}
If condition \eqref{eq3.2} is violated, then by our choices
\begin{equation*}
\frac{M_{+}(k,r,\theta)-(v_{+}-2^{-s_{0}}\omega\,\lambda(r))_{-}+2^{-s_{1}}\omega\,\lambda(r)}
{r}\geqslant \frac{\omega\,\lambda(r)}{2^{s_{1}}r}
\geqslant \frac{1+b_{0}}{r^{\,\delta/\overline{\delta}}}
\geqslant \frac{1+b_{0}}{R_{0}^{\,\delta}}.
\end{equation*}
Therefore, conditions $(\Psi_{1})$ and $(\Psi_{2})$ are applicable.
In the case $(\Psi_{1})$ we have
\begin{multline*}
\psi\left( x,\tau,
\frac{M_{+}(k,r,\theta)-(v_{+}-2^{-s_{0}}\omega\,\lambda(r))_{-}+2^{-s_{1}}\omega\,\lambda(r)}
{r} \right)
\\
\leqslant c_{2}\,\psi\left( x_{0},t_{0},
\frac{M_{+}(k,r,\theta)-(v_{+}-2^{-s_{0}}\omega\,\lambda(r))_{-}+2^{-s_{1}}\omega\,\lambda(r)}
{r} \right)  \hskip 15mm
\\
\leqslant c_{2}\,\psi\left( x_{0},t_{0}, \frac{\omega\,\lambda(r)}{2^{s_{0}-1}r} \right)
\leqslant 2^{q-2}c_{2}\, \psi\left( x_{0},t_{0}, \frac{\omega\,\lambda(r)}{2^{s_{0}}r} \right),
\ \ (x,\tau)\in B_{r}(x_{0})\times(\overline{t}-\theta, t).
\end{multline*}
Similarly, in the case $(\Psi_{2})$ we obtain
\begin{multline*}
\psi\left( x,\tau,
\frac{M_{+}(k,r,\theta)-(v_{+}-2^{-s_{0}}\omega\,\lambda(r))_{-}+2^{-s_{1}}\omega\,\lambda(r)}
{r} \right)
\\
\leqslant c_{2}\,\psi\left( x_{0},t_{0},
\frac{M_{+}(k,r,\theta)-(v_{+}-2^{-s_{0}}\omega\,\lambda(r))_{-}+2^{-s_{1}}\omega\,\lambda(r)}
{r} \right)
\\
\leqslant c_{2}\,\psi\left( x_{0},t_{0}, \frac{\omega\,\lambda(r)}{2^{s_{1}}r} \right),
\ \ (x,\tau)\in B_{r}(x_{0})\times(\overline{t}-\theta, t). \hskip 40mm
\end{multline*}
So, by our choices of $\theta$ we estimate the second term on the right-hand side of \eqref{eq3.4}
as follows:
\begin{equation}\label{eq3.6}
I_{3}\leqslant \gamma\sigma^{-q}(s_{1}-s_{0})\,|B_{r}(x_{0})|.
\end{equation}
Now we estimate from below the integral on the left-hand side of \eqref{eq3.4}, we have
$$
I_{1}\geqslant (s_{1}-s_{0}-2)^{2}\ln^{2}2\,
\left|\left\{ x\in B_{r}(x_{0}):v_{+}(x,t)
\leqslant 2^{-s_{1}}\omega\,\lambda(r) \right\}  \right|.
$$
Collecting \eqref{eq3.4}--\eqref{eq3.6}, we arrive at
\begin{multline*}
\left|\left\{ x\in B_{r}(x_{0}):v_{+}(x,t)
\leqslant 2^{-s_{1}}\omega\,\lambda(r) \right\}  \right|
\\
\leqslant
\left( n\sigma+ \left(\frac{s_{1}-s_{0}}{s_{1}-s_{0}-2}\right)^{2}(1-\beta_{0})
+\frac{\gamma}{\sigma^{q}}\, \frac{s_{1}-s_{0}}{(s_{1}-s_{0}-2)^{2}}  \right)
|B_{r}(x_{0})|, \ \ \text{for all} \ t\in(\overline{t}-\theta, \overline{t}).
\end{multline*}
Choosing $\sigma$ such that $n\sigma\leqslant \beta_{0}^{2}/4$ and then $s_{1}$
such that
$$
\left(\dfrac{s_{1}-s_{0}}{s_{1}-s_{0}-2}\right)^{2}\leqslant 1+\beta_{0},
\quad
\frac{\gamma}{\sigma^{q}}\, \frac{s_{1}-s_{0}}{(s_{1}-s_{0}-2)^{2}}
\leqslant \frac{1}{4}\,\beta_{0}^{2},
$$
we arrive at the required \eqref{eq3.3}, which completes the proof of the lemma.
\end{proof}

The next lemma involves the initial data.
\begin{lemma}\label{lem3.4}
Let $u\in \mathcal{B}_{1,g}(\Omega_{T})$ and let conditions ${\rm (g_{1})}$,
${\rm (g_{2})}$ and ${\rm (\Psi_{1})}$ be fulfilled in $Q_{R_{0}, R_{0}}(x_{0},t_{0})$.
Assume also that
\begin{equation}\label{eq3.7}
v_{\pm}(x, \overline{t}-\theta)\geqslant \frac{\omega\,\lambda(r)}{2^{s_{2}}},
\ \ x\in B_{r}(x_{0}),
\end{equation}
for some $s_{2}\geqslant 1+\log M$. Then for every $\nu\in(0,1)$ there exists number
$s_{3}>s_{2}$ depending only on the data, $s_{2}$, $\nu$, $\theta$, $\omega$ and $r$
such that either
\begin{equation}\label{eq3.8}
\omega\leqslant 2^{s_{3}}(1+b_{0})\,
\frac{r^{1-\frac{\delta}{\overline{\delta}}}}{\lambda(r)},
\end{equation}
or
\begin{equation}\label{eq3.9}
\left|\left\{ x\in B_{r/2}(x_{0}):v_{\pm}(x,t)
\leqslant \frac{\omega\,\lambda(r)}{2^{s_{3}}} \right\}  \right|
\leqslant \nu\,|B_{r/2}(x_{0})|,
\end{equation}
for all $t\in(\overline{t}-\theta,\overline{t})$.
\end{lemma}
\begin{proof}
The proof is similar to that of Lemma \ref{lem3.3}. By \eqref{eq3.7} the first term on the
right-hand side of \eqref{eq3.4} is equal to zero. So, repeating the same arguments as in the
previous proof we obtain
$$
\left|\left\{ x\in B_{r/2}(x_{0}):v_{\pm}(x,t)
\leqslant \frac{\omega\,\lambda(r)}{2^{s_{2}}} \right\}  \right|
\leqslant \frac{\gamma(s_{3}-s_{2})}{(s_{3}-s_{2}-2)^{2}}\,
\frac{\theta}{r^{2}}\,\psi\left(x_{0},t_{0}, \frac{\omega\,\lambda(r)}{2^{s_{0}}}  \right).
$$
Choosing $s_{3}$ from the condition
$$
\frac{\gamma(s_{3}-s_{2})}{(s_{3}-s_{2}-2)^{2}}\,
\frac{\theta}{r^{2}}\,\psi\left(x_{0},t_{0}, \frac{\omega\,\lambda(r)}{2^{s_{2}}} \right)
\leqslant \nu,
$$
we arrive at the required \eqref{eq3.9}, which proves the lemma.
\end{proof}

%%%%%%%%%%%%%%%%%%%%%%%%%%%%%%%%%%%%%%%%%%%%%%%%%%%%%%%%%%%%%%%%%%%%%%%%%%%%%%%%%%%%%%%%%%%%%%%%%%%%%%%%%%%%%%%%%%%%%%%
%%%%%%%%%%%%%%%%%%%%%%%%%%%%%%%%%%%%%%%%%%%%%%%%%%%%%%%%%%%%%%%%%%%%%%%%%%%%%%%%%%%%%%%%%%%%%%%%%%%%%%%%%%%%%%%%%%%%%%%

\subsection{De\,Giorgi type lemmas}\label{subSect3.3}

The next two lemmas will be used in the sequel, they are consequence of the Sobolev embedding
theorem and inequalities \eqref{eq2.1}--\eqref{eq2.3}.
\begin{lemma}\label{lem3.5}
Let $u\in \mathcal{B}_{1,g}(\Omega_{T})$ and let conditions ${\rm (g_{1})}$,
${\rm (g_{2})}$, ${\rm (\Psi_{1})}$ or ${\rm (\Psi_{2})}$ be fulfilled in
$Q_{R_{0}, R_{0}}(x_{0},t_{0})$. Fix $\xi\in (0, \frac{1}{2M})$, then there exists
$\nu\in(0,1)$ depending only on the data, $\xi$, $\omega$, $r$ and $\theta$ such that if
\begin{equation}\label{eq3.10}
\left| \left\{ (x,t)\in Q_{r,\theta}(x_{0},\overline{t}):
v_{\pm}(x,t)\leqslant \xi\,\omega\,\lambda(r)  \right\} \right|
\leqslant\nu\, |Q_{r,\theta}(x_{0},\overline{t})|,
\end{equation}
then either
\begin{equation}\label{eq3.11}
\xi\,\omega\leqslant4(1+b_{0})\,\frac{r^{1-\frac{\delta}{\overline{\delta}}}}{\lambda(r)},
\end{equation}
or
\begin{equation}\label{eq3.12}
v_{\pm}(x,t)\geqslant \frac{\xi\,\omega\,\lambda(r)}{4}
\quad \text{for a.a.} \ (x,t)\in Q_{r/2,\theta/2}(x_{0},\overline{t}).
\end{equation}
\end{lemma}
\begin{proof}
Further we will suppose that
\begin{equation}\label{eq3.13}
\xi\,\omega\geqslant4(1+b_{0})\,\frac{r^{1-\frac{\delta}{\overline{\delta}}}}{\lambda(r)}.
\end{equation}

For $j=0,1,2, \ldots,$ we define the sequences
$$
r_{j}:=\dfrac{r}{2}(1+2^{-j}), \quad \theta_{j}:=\dfrac{\theta}{2}(1+2^{-j}), \quad
\overline{r}_{j}:=\dfrac{r_{j}+r_{j+1}}{2}, \quad
\overline{\theta}_{j}:=\dfrac{\theta_{j}+\theta_{j+1}}{2},
$$
$$
B_{j}:=B_{r_{j}}(x_{0}), \quad \overline{B}_{j}:=B_{\overline{r}_{j}}(x_{0}),
\quad Q_{j}:=Q_{r_{j}, \theta_{j}}(x_{0},\overline{t}), \quad
\overline{Q}_{j}:=Q_{\overline{r}_{j}, \overline{\theta}_{j}}(x_{0},\overline{t}),
$$
$$
k_{j}:=\mu_{+}-\frac{\xi\,\omega\,\lambda(r)}{2}-\frac{\xi\,\omega\,\lambda(r)}{2^{j+1}},
\quad A_{j,k_{j}}:= Q_{j}\cap \{u>k_{j}\}, \quad
\overline{A}_{j,k_{j}}:= \overline{Q}_{j}\cap \{u>k_{j}\}.
$$
Let $\zeta_{j}\in C_{0}^{\infty}(\overline{B}_{j})$, $0\leqslant\zeta_{j}\leqslant1$,
$\zeta_{j}=1$ in $B_{j+1}$ and $|\nabla \zeta_{j}|\leqslant \gamma2^{j}/r$.
Consider also the function $\chi_{j}(t)=1$ for $t\geqslant \overline{t}-\theta_{j+1}$,
$\chi_{j}(t)=0$ for $t< \overline{t}-\theta_{j}$, $0\leqslant\chi_{j}(t)\leqslant1$
and $|\chi'_{j}|\leqslant \gamma2^{j}/\theta$.

We will assume that
\begin{equation}\label{eq3.14}
\esssup\limits_{Q_{r/2, \theta/2}(x_{0},\overline{t})}(u-k_{\infty})_{+}
\geqslant \frac{\xi\,\omega\,\lambda(r)}{4},
\end{equation}
because in the opposite case, inequality \eqref{eq3.12} is evident.
Using the fact that $$\esssup\limits_{Q_{j}}(u-k_{j})_{+}\leqslant\lambda(r),$$
by ${\rm (g_{1})}$ and ${\rm (g_{2})}$ we obtain for $(x,t)\in Q_{j}$
$$
\begin{aligned}
&g\left( x,t, \esssup\limits_{Q_{j}}\frac{(u-k_{j})_{+}}{r} \right)
\leqslant g\left( x,t, \frac{\xi\,\omega\,\lambda(r)}{r} \right)
\\
&\leqslant c_{2}\,g\left( x_{0},t_{0}, \frac{\xi\,\omega\,\lambda(r)}{r} \right)
\leqslant c_{2}\,4^{q-1}\,g\left( x_{0},t_{0}, \frac{\xi\,\omega\,\lambda(r)}{4r} \right)
\\
&\leqslant c^{2}_{2}\,4^{q-1}\,g\left( x,t, \frac{\xi\,\omega\,\lambda(r)}{4r} \right)
\leqslant c^{2}_{2}\,4^{q-1}\,g\left( x,t, \esssup\limits_{Q_{j}}\frac{(u-k_{j})_{+}}{r} \right).
\end{aligned}
$$
By this, inequalities \eqref{eq2.1}, \eqref{eq2.3} with $\varepsilon=1$ can be rewritten in the
form
$$
\iint\limits_{\overline{A}_{j,k_{j}}}
|\nabla (\min\{u,k_{j+1}\}-k_{j})_{+}|\,\zeta_{j}^{\,c_{2}}\chi_{j}\,dxdt
\leqslant
\gamma2^{j\gamma}\,\frac{\xi\,\omega\,\lambda(r)}{r}
\left( 1+\frac{r^{2}}
{\theta\psi\left( x_{0},t_{0}, \frac{\xi\,\omega\,\lambda(r)}{r} \right)} \right)
|A_{j,k_{j}}|,
$$
\begin{multline*}
\sup\limits_{\overline{t}-\overline{\theta}_{j}<t<\overline{t}}\,
\int\limits_{\overline{B}_{j}\times\{t\}} (u-k_{j})_{+}^{2}\,\zeta_{j}^{\,c_{2}}\chi_{j}\,dx
\\
\leqslant
\gamma2^{j\gamma}\left(\frac{\xi\,\omega\,\lambda(r)}{r}\right)^{2}
\psi\left( x_{0},t_{0}, \frac{\xi\,\omega\,\lambda(r)}{r} \right)
\left( 1+\frac{r^{2}}
{\theta\psi\left( x_{0},t_{0}, \frac{\xi\,\omega\,\lambda(r)}{r} \right)} \right)
|A_{j,k_{j}}|.
\end{multline*}
By the Sobolev embedding theorem from the last two inequalities we obtain
$$
\begin{aligned}
(k_{j+1}-k_{j})\,&|A_{j+1,k_{j+1}}|
\leqslant
\iint\limits_{\overline{A}_{j,k_{j}}}
 (\min\{u,k_{j+1}\}-k_{j})_{+}\,\zeta_{j}^{\,c_{2}}\chi_{j}\,dxdt
\\
&\leqslant \gamma |A_{j,k_{j}}|^{\frac{2}{n+2}}
\left( \sup\limits_{\overline{t}-\overline{\theta}_{j}<t<\overline{t}}\,
\int\limits_{\overline{B}_{j}\times\{t\}} (u-k_{j})_{+}^{2}\,\zeta_{j}^{\,c_{2}}\chi_{j}\,dx
 \right)^{\frac{1}{n+2}}
\\
&\times
\left(\, \iint\limits_{\overline{A}_{j,k_{j}}}
\left|\nabla \left[(\min\{u,k_{j+1}\}-k_{j})_{+}\,\zeta_{j}^{\,c_{2}}\right]\right|
\chi_{j}\,dxdt \right)^{\frac{n}{n+2}}
\\
&\leqslant \gamma2^{j\gamma}\, \frac{\xi\,\omega\,\lambda(r)}{r}\,
\left[\psi\left( x_{0},t_{0}, \frac{\xi\,\omega\,\lambda(r)}{r} \right)\right]^{\frac{1}{n+2}}
\left( 1+\frac{r^{2}}
{\theta\,\psi\left( x_{0},t_{0}, \frac{\xi\,\omega\,\lambda(r)}{r} \right)} \right)
^{\frac{n+1}{n+2}},
\end{aligned}
$$
which implies
$$
y_{j+1}:=\frac{|A_{j+1,k_{j+1}}|}{|Q_{r,\theta}(x_{0},\overline{t})|}\leqslant
\left( \frac{\theta\,\psi\left( x_{0},t_{0}, \frac{\xi\,\omega\,\lambda(r)}{r} \right)}
{r^{2}} \right)^{\frac{1}{n+2}}
\left( 1+\frac{r^{2}}
{\theta\,\psi\left( x_{0},t_{0}, \frac{\xi\,\omega\,\lambda(r)}{r} \right)} \right)
^{\frac{n+1}{n+2}}y_{j}^{1+\frac{1}{n+2}}.
$$
From this by Lemma \ref{lem3.2} we obtain that
$\lim\limits_{j\rightarrow\infty}y_{j}=0$, provided that $y_{0}\leqslant\nu$, where $\nu$ is
chosen to satisfy
\begin{equation}\label{eq3.15}
\nu=\frac{\gamma^{-1}r^{2}}
{\theta\,\psi\left( x_{0},t_{0}, \frac{\xi\,\omega\,\lambda(r)}{r} \right)}
\left( 1+\frac{r^{2}}
{\theta\,\psi\left( x_{0},t_{0}, \frac{\xi\,\omega\,\lambda(r)}{r} \right)} \right)
^{-n-1},
\end{equation}
which proves the lemma.
\end{proof}
\begin{lemma}[De\,Giorgi type lemma involving initial data]\label{lem3.6}
Let $u\in \mathcal{B}_{1,g}(\Omega_{T})$ and let conditions ${\rm (g_{1})}$,
${\rm (g_{2})}$, ${\rm (\Psi_{1})}$ or ${\rm (\Psi_{2})}$ be fulfilled in
$Q_{R_{0},R_{0}}(x_{0},t_{0})$. Fix $\xi\in(0,\frac{1}{2M})$, the there exists
$\nu_{1}\in(0,1)$ depending only on the data, $\xi$, $\omega$, $r$ and $\theta$
such that if
\begin{equation}\label{eq3.16}
v_{\pm}(x,\overline{t}-\theta)\geqslant \xi\,\omega\,\lambda(r)
\quad \text{for } \ x\in B_{r}(x_{0}),
\end{equation}
and
\begin{equation}\label{eq3.17}
|\{(x,t)\in Q_{r,\theta}(x_{0},\overline{t}):
v_{\pm}(x,t)\leqslant \xi\,\omega\,\lambda(r)\}|
\leqslant \nu_{1}|Q_{r,\theta}(x_{0},\overline{t})|,
\end{equation}
then either \eqref{eq3.11} holds, or
\begin{equation}\label{eq3.18}
v_{\pm}(x,t)\geqslant \frac{\xi\,\omega\,\lambda(r)}{4}
\quad \text{for a.a.} \ (x,t)\in Q_{r/2,\theta/2}(x_{0},\overline{t}),
\end{equation}
provided that
$\theta\leqslant r^{2}/\psi\left( x_{0},t_{0}, \dfrac{\xi\,\omega\,\lambda(r)}{r} \right)$.
\end{lemma}
\begin{proof}
The proof is similar to that of Lemma \ref{lem3.5}. Taking $\chi(t)=1$, using inequalities
\eqref{eq2.1}, \eqref{eq2.3} and our choice of $\theta$, and repeating the same arguments as
in the previous proof, we prove the theorem with $\nu_{1}=\gamma^{-1}$.
%\begin{equation}\label{eq3.19}
%\nu_{1}=\gamma^{-1}.
%\end{equation}
\end{proof}

%%%%%%%%%%%%%%%%%%%%%%%%%%%%%%%%%%%%%%%%%%%%%%%%%%%%%%%%%%%%%%%%%%%%%%%%%%%%%%%%%%%%%%%%%%%%%%%%%%%%%%%%%%%%%%%%%%%%%%
%%%%%%%%%%%%%%%%%%%%%%%%%%%%%%%%%%%%%%%%%%%%%%%%%%%%%%%%%%%%%%%%%%%%%%%%%%%%%%%%%%%%%%%%%%%%%%%%%%%%%%%%%%%%%%%%%%%%%%
%%%%%%%%%%%%%%%%%%%%%%%%%%%%%%%%%%%%%%%%%%%%%%%%%%%%%%%%%%%%%%%%%%%%%%%%%%%%%%%%%%%%%%%%%%%%%%%%%%%%%%%%%%%%%%%%%%%%%%

\section{Continuity in the ''degenerate'' case}\label{Sect4}

Fix $(x_{0},t_{0})\in\Omega_{T}$ and consruct the cylinder
$Q_{R_{0},R_{0}}(x_{0},t_{0})\subset\Omega_{T}$ and assume that conditions
${\rm (g_{1})}$, ${\rm (g_{2})}$, ${\rm (\Psi_{1})}$ hold in the cylinder
$Q_{R_{0},R_{0}}(x_{0},t_{0})$. Let $\rho>0$ be such that $\rho<R_{0}^{\overline{\delta}}$,
$\overline{\delta}>1+\delta+\delta/\delta_{0}$, where $\delta_{0}$ is the number from condition
${\rm (g_{2})}$ and $\delta$ was defined in ${\rm (\Psi_{1})}$. We consruct the cylinder
$$
Q_{\rho}(x_{0},t_{0}):=B_{\rho}(x_{0})
\times\left( t_{0}-\frac{1+b_{0}}{g(x_{0},t_{0},1)}\rho^{2-\frac{\delta}{\overline{\delta}}},
t_{0}\right)\subset Q_{R_{0},R_{0}}(x_{0},t_{0}),
$$
and set $\mu_{+}:=\esssup\limits_{Q_{\rho}(x_{0},t_{0})}u$,
$\mu_{-}:=\essinf\limits_{Q_{\rho}(x_{0},t_{0})}u$, $\omega:=\mu_{+}-\mu_{-}$.

Fix sufficiently large positive number $s_{\ast}$ which will be specified later
depending only on the data. If
$\omega\geqslant2^{s_{\ast}}(1+b_{0})\rho^{1-\frac{\delta}{\overline{\delta}}}/\lambda(\rho)$,
then we have
$$
\psi\left(x_{0},t_{0},\frac{\omega\,\lambda(\rho)}{2^{s_{\ast}}\rho}\right)
\geqslant
\psi\left(x_{0},t_{0}, \frac{1+b_{0}}{\rho^{\delta/\overline{\delta}}} \right)
=\rho^{\delta/\overline{\delta}}\,\,
\frac{g\left(x_{0},t_{0},(1+b_{0})\rho^{-\delta/\overline{\delta}}\right)}
{1+b_{0}}\geqslant \rho^{\delta/\overline{\delta}}\,\,
\frac{g(x_{0},t_{0},1)}{1+b_{0}},
$$
and therefore
$$
Q_{\rho,\theta_{\ast}}(x_{0},t_{0})\subset Q_{\rho}(x_{0},t_{0}),
\quad \theta_{\ast}:=\frac{\rho^{2}}{\psi\left(x_{0},t_{0},
\dfrac{\omega\,\lambda(\rho)}{2^{s_{\ast}}\rho}\right)}.
$$

In $Q_{\rho, \theta_{\ast}}(x_{0},t_{0})$ consider the cylinders
$$
Q_{\rho,\eta}(x_{0},\overline{t}):=B_{\rho}(x_{0})\times(\overline{t}-\eta,\overline{t}),
\quad \eta=\frac{\rho^{2}}{4\,\psi\left(x_{0},t_{0}, \dfrac{\omega\,\lambda(\rho)}{2^{s_{0}}\rho}
 \right)},
 \quad t_{0}-\theta_{\ast}\leqslant \overline{t}-\eta<\overline{t}\leqslant t_{0},
$$
where $s_{0}$ is fixed by the condition $2^{s_{0}-1}\geqslant M$, $s_{0}<s_{\ast}$.

The following two alternative cases are possible.

\textsl{First alternative}. There exists a cylinder
$Q_{\rho,\eta}(x_{0},\overline{t})\subset Q_{\rho, \theta_{\ast}}(x_{0},t_{0})$
such that
$$
\left|\left\{(x,t)\in Q_{\rho,\eta}(x_{0},\overline{t}):
u(x,t)\leqslant\mu_{-}+\frac{\omega\,\lambda(\rho)}{2^{s_{0}}}  \right\}  \right|
\leqslant \nu\,|Q_{\rho,\eta}(x_{0},\overline{t})|,
$$
where $\nu$ is a sufficiently small positive number which will be chosen later
depending only upon the data.

\textsl{Second alternative}. For all cylinders
$Q_{\rho,\eta}(x_{0},\overline{t})\subset Q_{\rho, \theta_{\ast}}(x_{0},t_{0})$
the opposite inequality
$$
\left|\left\{(x,t)\in Q_{\rho,\eta}(x_{0},\overline{t}):
u(x,t)\leqslant\mu_{-}+\frac{\omega\,\lambda(\rho)}{2^{s_{0}}}  \right\}  \right|
> \nu\,|Q_{\rho,\eta}(x_{0},\overline{t})|
$$
holds.

Fix positive number $c_{\ast}$ which will be chosen depending only upon the data
and assume that
\begin{equation}\label{eq4.1}
\omega\geqslant c_{\ast}(1+b_{0})\,\frac{\rho^{1-\frac{\delta}{\overline{\delta}}}}
{\lambda(\rho)}.
\end{equation}

%%%%%%%%%%%%%%%%%%%%%%%%%%%%%%%%%%%%%%%%%%%%%%%%%%%%%%%%%%%%%%%%%%%%%%%%%%%%%%%%%%%%%%%%%%%%%%%%%%%%%%%%%%%%%%%%%%%%%%%%%
%%%%%%%%%%%%%%%%%%%%%%%%%%%%%%%%%%%%%%%%%%%%%%%%%%%%%%%%%%%%%%%%%%%%%%%%%%%%%%%%%%%%%%%%%%%%%%%%%%%%%%%%%%%%%%%%%%%%%%%%%

\subsection{Analysis of the first alternative}\label{subSect4.1}

Using Lemma \ref{lem3.5} in the cylinder $Q_{\rho,\eta}(x_{0},\overline{t})$ and
choosing $\nu$ from the condition \eqref{eq3.15} $\nu=\gamma^{-1}$, we obtain that
$$
u(x,\overline{t})\geqslant \mu_{-}+\frac{\omega\,\lambda(\rho)}{2^{s_{0}+2}}
\quad \text{for all} \ x\in B_{\rho/2}(x_{0}).
$$
Choosing $s_{3}$ from the condition
$$
\frac{2^{(s_{\ast}-s_{0}-2)(q-2)}}{s_{3}-s_{0}-2}
\leqslant \nu_{1},
$$
where $\nu_{1}$ is the number claimed by Lemma \ref{lem3.6}, and using the fact that
by ${\rm (g_{1})}$
$$
\frac{\psi\left(x_{0},t_{0}, \dfrac{\omega\,\lambda(\rho)}{2^{s_{0}+2}}\right)}
{\psi\left(x_{0},t_{0}, \dfrac{\omega\,\lambda(\rho)}{2^{s_{\ast}}}\right)}
\leqslant 2^{(s_{\ast}-s_{0}-2)(q-2)},
$$
from Lemmas \ref{lem3.4} and \ref{lem3.6} we conclude that if $c_{\ast}\geqslant2^{s_{3}}$,
then
$$
u(x,t)\geqslant \mu_{-}+\frac{\omega\,\lambda(\rho)}{2^{s_{3}}} \quad
\text{for a.a.} \ (x,t)\in Q_{\rho/8, \eta/8}(x_{0},t_{0}),
$$
which implies that
\begin{equation}\label{eq4.2}
\osc\limits_{Q_{\rho/8, \eta/8}(x_{0},t_{0})}u
\leqslant \left(1-2^{-s_{3}}\lambda(\rho)\right)\omega.
\end{equation}

%%%%%%%%%%%%%%%%%%%%%%%%%%%%%%%%%%%%%%%%%%%%%%%%%%%%%%%%%%%%%%%%%%%%%%%%%%%%%%%%%%%%%%%%%%%%%%%%%%%%%%%%%%%%%%%%%%%%%%%%%
%%%%%%%%%%%%%%%%%%%%%%%%%%%%%%%%%%%%%%%%%%%%%%%%%%%%%%%%%%%%%%%%%%%%%%%%%%%%%%%%%%%%%%%%%%%%%%%%%%%%%%%%%%%%%%%%%%%%%%%%%

\subsection{Analysis of the second alternative}\label{subSect4.2}

By the assumption of the second alternative, we have
\begin{equation}\label{eq4.3}
\left|\left\{ (x,t)\in Q_{\rho,\eta}(x_{0},\overline{t}):
u(x,t)\geqslant\mu_{+}-\frac{\omega\,\lambda(\rho)}{2^{s_{0}}} \right\} \right|
\leqslant (1-\nu)\,|Q_{\rho,\eta}(x_{0},\overline{t})|,
\end{equation}
for all cylinders
$Q_{\rho,\eta}(x_{0},\overline{t})\subset Q_{\rho,\theta_{\ast}}(x_{0},t_{0})$.

\textsl{Claim}. Fix a cylinder $Q_{\rho,\eta}(x_{0},\overline{t})$, then there exists
$\widetilde{t}\in\left( \overline{t}-\eta, \overline{t}-\eta\,\dfrac{\nu}{2} \right)$
such that
\begin{equation}\label{eq4.4}
\left|\left\{ x\in B_{\rho}(x_{0}):
u(x,\widetilde{t})\geqslant \mu_{+}- \frac{\omega\,\lambda(\rho)}{2^{s_{0}}} \right\} \right|
\leqslant \frac{1-\nu}{1-\nu/2}\, |B_{\rho}(x_{0})|.
\end{equation}

Suppose not. Then for all
$t\in \left( \overline{t}-\eta, \overline{t}-\eta\,\dfrac{\nu}{2} \right)$
there holds
$$
\left|\left\{ x\in B_{\rho}(x_{0}):
u(x,t)\geqslant \mu_{+}- \frac{\omega\,\lambda(\rho)}{2^{s_{0}}} \right\} \right|
> \frac{1-\nu}{1-\nu/2}\, |B_{\rho}(x_{0})|.
$$
Hence
\begin{multline*}
\left|\left\{ (x,t)\in Q_{\rho,\eta}(x_{0},\overline{t}):
u(x,t)\geqslant \mu_{+}- \frac{\omega\,\lambda(\rho)}{2^{s_{0}}} \right\}  \right|
\\
\geqslant
\int\limits_{\overline{t}-\eta}^{\overline{t}-\eta\,\frac{\nu}{2}}
\left|\left\{ x\in B_{\rho}(x_{0}):
u(x,t)\geqslant \mu_{+}- \frac{\omega\,\lambda(\rho)}{2^{s_{0}}} \right\} \right|dt
>(1-\nu)\,|Q_{\rho,\eta}(x_{0},\overline{t})|,
\end{multline*}
which contradicts \eqref{eq4.3}.

Inequality \eqref{eq4.4} together with Lemma \ref{lem3.2} imply
$$
\left|\left\{ x\in B_{\rho}(x_{0}):
u(x,t)\geqslant \mu_{+}- \frac{\omega\,\lambda(\rho)}{2^{s_{1}}} \right\} \right|
\leqslant \left( 1-\frac{1}{2}\bigg( \frac{\nu/2}{1-\nu} \bigg)^{2} \right)
|B_{\rho}(x_{0})| \leqslant \bigg( 1-\frac{\nu^{2}}{8} \bigg)\,|B_{\rho}(x_{0})|,
$$
for all $t\in(\widetilde{t}, \widetilde{t}+\eta)$ and $s_{1}$ is the number defined
in Lemma \ref{lem3.2}.

Since \eqref{eq4.3} holds true for all cylinders $Q_{\rho,\eta}(x_{0},\overline{t})$,
from the previous inequality we obtain
\begin{equation}\label{eq4.5}
\left|\left\{ x\in B_{\rho}(x_{0}):
u(x,t)\geqslant \mu_{+}- \frac{\omega\,\lambda(\rho)}{2^{s_{1}}} \right\} \right|
\leqslant \bigg( 1-\frac{\nu^{2}}{8} \bigg)\,|B_{\rho}(x_{0})|,
\end{equation}
for all $t\in(t_{0}-\theta_{\ast}, t_{0})$.

\begin{lemma}\label{lem4.1}
For any $\nu\in(0,1)$ there exists a number $s_{\ast}>s_{1}$, depending on the data only,
such that
\begin{equation}\label{eq4.6}
\left|\left\{ (x,t)\in Q_{\rho, \theta_{\ast}}(x_{0},t_{0}):
u(x,t)\geqslant \mu_{+}- \frac{\omega\,\lambda(\rho)}{2^{s_{\ast}+1}} \right\} \right|
\leqslant \nu\, |Q_{\rho, \theta_{\ast}}(x_{0},t_{0})|.
\end{equation}
\end{lemma}
\begin{proof}
For $s=s_{1}+1, \ldots, s_{\ast}$ set $k_{s}:=\mu_{+}-2^{-s}\omega\,\lambda(\rho)$ and use
inequality \eqref{eq2.1} in the cylinder $Q_{3\rho/2, 3\theta_{\ast}/2}(x_{0},t_{0})$ with
$\sigma=1/2$ and choose $\chi(t)$ such that $0\leqslant\chi(t)\leqslant1$, $\chi(t)=0$ for
$t\leqslant t_{0}-\frac{3}{2}\theta_{\ast}$, $\chi(t)=1$ for $t\geqslant t_{0}-\theta_{\ast}$
and $|\chi'(t)|\leqslant 2/\theta_{\ast}$.

We also assume that for every $s_{1}+1\leqslant s\leqslant s_{\ast}$
\begin{equation}\label{eq4.7}
\esssup\limits_{Q_{\rho, \theta_{\ast}}(x_{0},t_{0})}(u-k_{s})_{+}
\geqslant \frac{\omega\,\lambda(\rho)}{2^{s+1}},
\end{equation}
because, if for some $s_{1}+1\leqslant \overline{s}\leqslant s_{\ast}$
$$
\esssup\limits_{Q_{\rho, \theta_{\ast}}(x_{0},t_{0})}(u-k_{\overline{s}})_{+}
\leqslant \frac{\omega\,\lambda(\rho)}{2^{\overline{s}+1}},
$$
then inequality \eqref{eq4.6} is evident.

Using the fact that
$\esssup\limits_{Q_{3\rho/2, 3\theta_{\ast}/2}(x_{0},t_{0})}(u-k_{s})_{+}
\leqslant \dfrac{\omega\,\lambda(\rho)}{2^{s}}$, by conditions ${\rm (g_{1})}$,
${\rm (g_{2})}$ we have for $(x,t)\in Q_{3\rho/2, 3\theta_{\ast}/2}(x_{0},t_{0})$
$$
g\left( x,t, \frac{(u-k_{s})_{+}}{\rho} \right)\leqslant \gamma
g\left( x_{0},t_{0}, \frac{\omega\,\lambda(\rho)}{2^{s}\rho} \right)
\leqslant \gamma g\left( x,t, \frac{(u-k_{s})_{+}}{\rho} \right).
$$
By this, inequality \eqref{eq2.1} can be rewritten in the form
$$
\begin{aligned}
\iint\limits_{A^{+}_{k_{s+1},\rho, \theta_{\ast}}\setminus A^{+}_{k_{s},\rho, \theta_{\ast}}}
|\nabla u|\,dxdt
&\leqslant \frac{\gamma}{\varepsilon}\, \frac{\omega\,\lambda(\rho)}{2^{s}\rho}\,
|A^{+}_{k_{s+1},3\rho/2, 3\theta_{\ast}/2}\setminus A^{+}_{k_{s},3\rho/2, 3\theta_{\ast}/2}|
\\
&+\gamma\varepsilon^{\,p-1}\,\frac{\omega\,\lambda(\rho)}{2^{s}\rho}
\left( 1+ \frac{\psi\left( x_{0},t_{0}, \dfrac{\omega\,\lambda(\rho)}{2^{s_{\ast}}\rho} \right)}
{\psi\left( x_{0},t_{0}, \dfrac{\omega\,\lambda(\rho)}{2^{s}\rho} \right)} \right)
|Q_{\rho, \theta_{\ast}}(x_{0},t_{0})|.
\end{aligned}
$$

If $c_{\ast}\geqslant 2^{s_{\ast}+1}$, then by \eqref{eq4.1}
$$
\frac{\omega\,\lambda(\rho)}{2^{s_{\ast}}\rho}
\geqslant \frac{c_{\ast}(1+b_{0})}{2^{s_{\ast}}}\,
\rho^{-\delta/\overline{\delta}}\geqslant (1+b_{0})R_{0}^{-\delta},
$$
therefore condition ${\rm (\Psi_{1})}$ is applicable, and from the previous we obtain
\begin{multline*}
\iint\limits_{A^{+}_{k_{s+1},\rho,\theta_{\ast}}\setminus A^{+}_{k_{s},\rho,\theta_{\ast}}}
|\nabla u|\,dxdt
\\
\leqslant
\frac{\gamma}{\varepsilon}\,\frac{\omega\,\lambda(\rho)}{2^{s}\rho}\,
|A^{+}_{k_{s+1},3\rho/2,3\theta_{\ast}/2}\setminus A^{+}_{k_{s},3\rho/2,3\theta_{\ast}/2}|
+\gamma\varepsilon^{\,p-1}\,\frac{\omega\,\lambda(\rho)}{2^{s}\rho}\,
|Q_{\rho,\theta_{\ast}}(x_{0},t_{0})|.
\end{multline*}
Choosing $\varepsilon$ from the condition
$$
\varepsilon=
\left( \frac{|A^{+}_{k_{s+1},3\rho/2,3\theta_{\ast}/2}\setminus
A^{+}_{k_{s},3\rho/2,3\theta_{\ast}/2}|}
{|Q_{3\rho/2,3\theta_{\ast}/2}(x_{0},t_{0})|} \right)^{1/p},
$$
from this we get
\begin{equation}\label{eq4.8}
\iint\limits_{A^{+}_{k_{s+1},\rho,\theta_{\ast}}\setminus A^{+}_{k_{s},\rho,\theta_{\ast}}}
|\nabla u|\,dxdt
\leqslant \gamma\varepsilon^{\,p-1}\,\frac{\omega\,\lambda(\rho)}{2^{s}\rho}\,
|Q_{\rho,\theta_{\ast}}(x_{0},t_{0})|.
\end{equation}
Using Lemma \ref{lem3.1} with $k=k_{s}$, $l=k_{s+1}$, due to \eqref{eq4.5} we obtain the
inequality
$$
\frac{\omega\,\lambda(\rho)}{2^{s+1}}\,|A^{+}_{k_{s+1},\rho}(t)|
\leqslant \gamma(\nu)\,\rho
\int\limits_{A^{+}_{k_{s+1},\rho}(t)\setminus A^{+}_{k_{s},\rho}(t)}
|\nabla u|\,dx,
$$
for all $t\in(t_{0}-\theta_{\ast},t_{0})$, where
$A^{+}_{k,\rho}(t):=\{x\in B_{\rho}(x_{0}):u(x,t)>k\}$.
Integrating tha last inequality with respect to $t\in(t_{0}-\theta_{\ast},t_{0})$
and using \eqref{eq4.8}, we obtain
$$
\left( \frac{|A^{+}_{k_{s+1},\rho,\theta_{\ast}}|}{|Q_{\rho,\theta_{\ast}}(x_{0},t_{0})|} \right)
^{\frac{p}{p-1}}\leqslant \gamma(\nu)\,
\frac{|A^{+}_{k_{s+1},3\rho/2,3\theta_{\ast}/2}\setminus
A^{+}_{k_{s},3\rho/2,3\theta_{\ast}/2}|}
{|Q_{3\rho/2,3\theta_{\ast}/2}(x_{0},t_{0})|}.
$$
Summing up the last inequality in $s$, $s_{1}+1\leqslant s\leqslant s_{\ast}$,
we conclude that
$$
|A^{+}_{k_{s+1},\rho,\theta_{\ast}}|\leqslant
\frac{\gamma(\nu)}{(s_{\ast}-s_{1}-1)^{\frac{p-1}{p}}}\,
|Q_{\rho,\theta_{\ast}}(x_{0},t_{0})|.
$$
Choosing $s_{\ast}$ by the condition
$\dfrac{\gamma(\nu)}{(s_{\ast}-s_{1}-1)^{\frac{p-1}{p}}}
\leqslant\nu$, we obtain inequality \eqref{eq4.6}, which proves
Lemma \ref{lem4.1}.
\end{proof}

Using Lemma \ref{lem3.5}, we obtain
$u(x,t)\leqslant \mu_{+}- \dfrac{\omega\,\lambda(\rho)}{2^{s_{\ast}+3}}$
for a.a. $(x,t)\in Q_{\rho/2,\theta_{\ast}/2}(x_{0},t_{0})$,
%$$
%u(x,t)\leqslant \mu_{+}- \frac{\omega\,\lambda(\rho)}{2^{s_{\ast}+3}}
%\quad \text{for a.a.} \ (x,t)\in Q_{\rho/2,\theta_{\ast}/2}(x_{0},t_{0}).
%$$
which implies that
\begin{equation}\label{eq4.9}
\osc\limits_{Q_{\rho/2,\theta_{\ast}/2}(x_{0},t_{0})} u
\leqslant\left(1-\frac{\lambda(\rho)}{2^{s_{\ast}+3}}\right)\omega.
\end{equation}
Collecting \eqref{eq4.2} and \eqref{eq4.9}, we obtain that
\begin{equation}\label{eq4.10}
\osc\limits_{Q_{\rho/8,\theta_{\ast}/8}(x_{0},t_{0})} u
\leqslant\left(1-\frac{\lambda(\rho)}{2^{s_{3}+1}}\right)\omega,
\end{equation}
where $s_{3}$ was defined in \eqref{eq4.2} depending only on the data.

For $j=0,1,2,\ldots$ define the sequences
$$
r_{j}:=c^{-j}\rho, \quad
\theta_{j}:=\dfrac{r_{j}^{2}}
{\psi\left(x_{0},t_{0}, \dfrac{\omega_{j}\lambda(r_{j})}{r_{j}}\right)},
\quad
Q_{j}:=Q_{r_{j}, \theta_{j}}(x_{0},t_{0}),
$$
where $c>1$ will be chosen depending only on the known data, and
$$
\omega_{j+1}:=\max\left\{\left(1-\frac{\lambda(r_{j})}{2^{s_{3}+1}}\right)\omega_{j}, \
c_{\ast}(1+b_{0})\,\frac{r_{j}^{1-\frac{\delta}{\overline{\delta}}}}{\lambda(r_{j})} \right\},
 \quad \omega_{0}:=\omega.
$$

Note the inequality
\begin{equation}\label{eq4.11}
\lambda(\rho_{1})\leqslant
2\left( \frac{\rho_{1}}{\rho_{2}} \right)^{1-\delta_{0}}
\lambda(\rho_{2}), \quad
0<\rho_{2}<\rho_{1}<R_{0}/2,
\end{equation}
which is simple consequence of our choice of $\lambda(r)$.

If inequality \eqref{eq4.1} holds, then by \eqref{eq4.11} we have
$$
\begin{aligned}
\omega_{j+1}
&\geqslant \left( 1-\frac{\lambda(r_{j})}{2^{s_{3}+1}} \right)\omega_{j}
\geqslant \left( 1-\frac{1}{2^{s_{3}+1}} \right)\omega_{j}
\\
&\geqslant
\left( 1-\frac{1}{2^{s_{3}+1}} \right)^{j+1}\omega
\geqslant
c_{\ast}(1+b_{0})\,\frac{\rho^{1-\frac{\delta}{\overline{\delta}}}}{\lambda(\rho)}
\left( 1-\frac{1}{2^{s_{3}+1}} \right)^{j+1}
\\
&\geqslant
c_{\ast}(1+b_{0})\,\frac{r_{j+1}^{1-\frac{\delta}{\overline{\delta}}}}{\lambda(r_{j+1})}
\left( \frac{1}{2}\,c^{\delta_{0}-\frac{\delta}{\overline{\delta}}}
\left(1-\frac{1}{2^{s_{3}+1}}\right) \right)^{j+1}.
\end{aligned}
$$
So, if
$c\geqslant \left(2\left(1-\dfrac{1}{2^{s_{3}+1}}\right)^{-1}\right)
^{\frac{1}{\delta_{0}-\delta/\overline{\delta}}}$,
then from the previous inequality we obtain
$$
\omega_{j+1}\geqslant c_{\ast}(1+b_{0})\,
\frac{r_{j+1}^{1-\frac{\delta}{\overline{\delta}}}}{\lambda(r_{j+1})},
\quad j=0,1,2,\ldots .
$$
Moreover, if
$c\geqslant \left(2\left(1-\dfrac{1}{2^{s_{3}+1}}\right)^{-1}\right)
^{1/\delta_{0}}$,
then $\theta_{j+1}\leqslant \theta_{j}$ and $Q_{j+1}\subset Q_{j}$,
$j=0,1,2,\ldots$\,. Therefore, if
$c=8\left(2\left(1-\dfrac{1}{2^{s_{3}+1}}\right)^{-1}\right)
^{\frac{1}{\delta_{0}-\delta/\overline{\delta}}}$,
by \eqref{eq4.10} we have $\osc\limits_{Q_{1}}u\leqslant \omega_{1}$.

Repeating the previous procedure, by our choices we obtain that
$\osc\limits_{Q_{j}}u\leqslant \omega_{j}$, $j=0,1,2,\ldots$\,.
Iterating this inequality, we have for any $j\geqslant1$
\begin{equation}\label{eq4.12}
\begin{aligned}
\osc\limits_{Q_{j}}u
&\leqslant \omega\prod\limits_{i=0}^{j-1}
\left(1-\frac{\lambda(r_{i})}{2^{s_{3}+1}}  \right)+c_{\ast}(1+b_{0})
\sum\limits_{i=0}^{j-1}
\frac{r_{i}^{1-\frac{\delta}{\overline{\delta}}}}{\lambda(r_{i})}
\\
&\leqslant \omega \exp\left( -2^{-s_{3}-1}\sum\limits_{i=0}^{j-1}\lambda(r_{i})\right)
+c_{\ast}(1+b_{0})\sum\limits_{i=0}^{j-1} \frac{r_{i}^{1-\delta_{0}}}{\lambda(r_{i})}
\\
&\leqslant \omega \exp\left(-\gamma\int_{r_{j}}^{c\rho} \lambda(s)\,\frac{ds}{s}\right)
+c_{\ast}(1+b_{0})\, \frac{\rho^{1-\delta_{0}}}{\lambda(\rho_{0})}\,\sum\limits_{i=0}^{j-1}
\left( \frac{3}{4} \right)^{(1-\delta_{0})i}
\\
&\leqslant \omega \exp\left(-\gamma\int_{r_{j}}^{c\rho} \lambda(s)\,\frac{ds}{s}\right)
+\gamma(1+b_{0})\, \frac{\rho^{1-\delta_{0}}}{\lambda(\rho_{0})}.
\end{aligned}
\end{equation}
By our choices
$\dfrac{\omega_{j}\lambda(r_{j})}{r_{j}}
\geqslant c_{\ast}(1+b_{0})\,R_{0}^{-\delta}$, $j=1,2, \ldots$\,.
So, by ${\rm (\Psi_{1})}$ and ${\rm (g_{1})}$ we obtain
$$
\theta_{j}\geqslant \frac{r_{j}^{2}}{\psi(x_{0},t_{0},2M/r_{j})}\geqslant
\frac{r_{j}^{q}(2M)^{2-q}}{g(x_{0},t_{0},1)}=\widetilde{\theta}_{j}, \quad
j\geqslant1.
$$
Therefore, inequality \eqref{eq4.12} implies
$$
\osc\limits_{Q_{r_{j}, \widetilde{\theta}_{j}}(x_{0},t_{0})}u
\leqslant 2M
\exp\left(-\gamma\int_{r_{j}}^{\rho} \lambda(s)\,\frac{ds}{s}\right)
+\gamma(1+b_{0})\, \frac{\rho^{1-\delta_{0}}}{\lambda(\rho)},
$$
which yields the continuity of $u$ at $(x_{0},t_{0})$.
To complete the proof of Theorem \ref{th2.1} in the ''degenerate'' case,
note that if $p\geqslant2$, then, as it was mentioned in Section \ref{Introduction},
by (${\rm g}_{1}$) condition ($\Psi_{1}$) holds with $\delta=b_{0}=0$; therefore,
the number $R_{0}$ claimed in the definition of the cylinder $Q_{R_{0},R_{0}}(x_{0},t_{0})$
depends only on the distance between $(x_{0},t_{0})$ and $\partial\Omega_{T}$, so, in this case
$u\in C_{{\rm loc}}(\Omega_{T})$.

%This completes the proof of .

%%%%%%%%%%%%%%%%%%%%%%%%%%%%%%%%%%%%%%%%%%%%%%%%%%%%%%%%%%%%%%%%%%%%%%%%%%%%%%%%%%%%%%%%%%%%%%%%%%%%%%%%%%%%%%%%%%%%%%%%%%%
%%%%%%%%%%%%%%%%%%%%%%%%%%%%%%%%%%%%%%%%%%%%%%%%%%%%%%%%%%%%%%%%%%%%%%%%%%%%%%%%%%%%%%%%%%%%%%%%%%%%%%%%%%%%%%%%%%%%%%%%%%%
%%%%%%%%%%%%%%%%%%%%%%%%%%%%%%%%%%%%%%%%%%%%%%%%%%%%%%%%%%%%%%%%%%%%%%%%%%%%%%%%%%%%%%%%%%%%%%%%%%%%%%%%%%%%%%%%%%%%%%%%%%%
%%%%%%%%%%%%%%%%%%%%%%%%%%%%%%%%%%%%%%%%%%%%%%%%%%%%%%%%%%%%%%%%%%%%%%%%%%%%%%%%%%%%%%%%%%%%%%%%%%%%%%%%%%%%%%%%%%%%%%%%%%%

\section{Continuity in the ''singular'' case}\label{Sect5}

Fix $(x_{0},t_{0})\in \Omega_{T}$ and construct the cylinder
$Q_{R_{0},R_{0}}(x_{0},t_{0})\subset \Omega_{T}$ and assume that conditions
${\rm (g_{1})}$, ${\rm (g_{2})}$, ${\rm (\Psi_{2})}$ hold in the cylinder
$Q_{R_{0},R_{0}}(x_{0},t_{0})$. Let $\rho>0$ be such that
$\rho<R_{0}^{\overline{\delta}}$, $\overline{\delta}>1+\delta q+\delta/\delta_{0}$,
where $\delta_{0}$ is the number defined in condition ${\rm (g_{2})}$ and $\delta$ was
defined in ${\rm (\Psi_{2})}$ and construct the cylinder
$$
Q_{\rho}(x_{0},t_{0}):=B_{\rho}(x_{0})\times
\left( t_{0}-\frac{2M\rho}{g(x_{0},t_{0},1)}, \, t_{0} \right)
\subset Q_{R_{0},R_{0}}(x_{0},t_{0}),
$$
and set $\mu_{+}:=\esssup\limits_{Q_{\rho}(x_{0},t_{0})}u$,
$\mu_{-}:=\essinf\limits_{Q_{\rho}(x_{0},t_{0})}u$, $\omega:=\mu_{+}-\mu_{-}$.

Fix sufficiently large positive number $c_{\ast}$ which will be chosen later depending
only upon the data. If
\begin{equation}\label{eq5.1}
\omega\geqslant c_{\ast}(1+b_{0})\,
\frac{\rho^{1-\frac{\delta}{\overline{\delta}}}}{\lambda(\rho)},
\end{equation}
then
$$
\psi\left(x_{0},t_{0}, \frac{\omega\lambda(\rho)}{\rho}\right)
=\frac{\rho}{\omega\lambda(\rho)}\,
g\left(x_{0},t_{0}, \frac{\omega\lambda(\rho)}{\rho}\right)
\geqslant \frac{\rho\, g(x_{0},t_{0},1)}{2M},
$$
and the following inclusion is true $Q_{\rho, \theta}(x_{0},t_{0})\subset Q_{\rho}(x_{0},t_{0})$,
$\theta:= \dfrac{\rho^{2}}
{\psi\left( x_{0},t_{0},\eta\,\frac{\omega\lambda(\rho)}{\rho} \right)}$,
%$$
%Q_{\rho, \theta}(x_{0},t_{0})\subset Q_{\rho}(x_{0},t_{0}),
%\quad \theta:= \frac{\rho^{2}}
%{\psi\left( x_{0},t_{0},\eta\,\dfrac{\omega\lambda(\rho)}{\rho} \right)},
%$$
where $\eta$ is a sufficiently small positive number depending only on the
known data, which will be specified later.

The following two alternative cases are possible:
$$
\left| \left\{x\in B_{\rho}(x_{0}): u(x,t_{0}-\theta)\leqslant\mu_{-}+
\frac{\omega\lambda(\rho)}{2^{s_{0}}}  \right\} \right|
\leqslant \frac{1}{2}\,|B_{\rho}(x_{0})|,
$$
or
$$
\left| \left\{x\in B_{\rho}(x_{0}): u(x,t_{0}-\theta)\geqslant\mu_{+}-
\frac{\omega\lambda(\rho)}{2^{s_{0}}}  \right\} \right|
\leqslant \frac{1}{2}\,|B_{\rho}(x_{0})|,
$$
where the number $s_{0}$ is fixed by the condition $2^{s_{0}-1}\geqslant M$.
Both alternative cases can be considered completely similar and assume, for example,
the first one.

Further we will also assume that inequality \eqref{eq5.1} holds.

Lemma \ref{lem3.3} with
$
\theta= \dfrac{\rho^{2}}
{\psi\left( x_{0},t_{0},\eta\,\frac{\omega\lambda(\rho)}{\rho} \right)}
$
implies that
\begin{equation}\label{eq5.2}
\left|\left\{ x\in B_{\rho}(x_{0}):
u(x,t)\leqslant\mu_{-}+\eta\,\omega\lambda(\rho) \right\} \right|
\leqslant\frac{7}{8}\,\,|B_{\rho}(x_{0})| \quad
\text{for all} \ t\in(t_{0}-\theta, t_{0}).
\end{equation}

For the function $v_{-}=u-\mu_{-}$ we will use inequality \eqref{eq2.5} with
$\sigma=1/2$, $r=\rho$, $k=\varepsilon^{j}\omega\lambda(\rho)$,
$j=1,2,\ldots,j_{\ast}$, where $\varepsilon\in(0,1)$ and $j_{\ast}>1$ will be determined
later depending only on the  data. Set %(tut nado skazat o funccii $\zeta$)
$$
A_{j}(t):=\{x\in B_{\rho}(x_{0}): v_{-}(x,t)\leqslant \varepsilon^{j}\omega\lambda(\rho)\},
$$
$$
Y_{j}(t):=\frac{1}{|B_{\rho}(x_{0})|} \int\limits_{A_{j}(t)}
\frac{t-t_{0}+\theta}{\theta}\,\zeta^{\,c_{2}}\,dx, \quad
y_{j}:=\sup\limits_{t_{0}-\frac{\theta}{2}<t<t_{0}} Y_{j}(t).
$$
By ${\rm (g_{2})}$ we have for any $(x,t)\in Q_{\rho,\theta}(x_{0},t_{0})$
$$
g(x,t,w_{k,\varepsilon}/\rho)\leqslant c_{2}g(x_{0},t_{0},w_{k,\varepsilon}/\rho)
\leqslant c_{2}^{2}g(x,t,w_{k,\varepsilon}/\rho),
$$
where $w_{k,\varepsilon}=(1+\varepsilon)k-(v_{-}-k)_{-}$, $k=\varepsilon^{j}\omega\lambda(\rho)$.
Therefore, we estimate the second term on the right-hand side of \eqref{eq2.5} as follows:
\begin{equation}\label{eq5.3}
\gamma(1+\varepsilon_{1}^{-p})\int\limits_{B_{\rho}(x_{0})\times\{t\}}
\frac{g(x,t,w_{k,\varepsilon}/\rho)}{g(x_{0},t_{0},w_{k,\varepsilon}/\rho)}\,
\zeta^{\,c_{2}-q}\,dx\leqslant \gamma(1+\varepsilon_{1}^{-p})\,|B_{\rho}(x_{0})|.
\end{equation}
By ${\rm (g_{1})}$ we have that
$$
\begin{aligned}
\int\limits_{0}^{(v_{-}-k)_{-}}\frac{(1+\varepsilon)k-s}
{\mathcal{G}\left( x_{0},t_{0}, \frac{(1+\varepsilon)k-s}{\rho} \right)}\,dx
&\leqslant \gamma\rho \int\limits_{0}^{(v_{-}-k)_{-}}
\frac{ds}{g\left( x_{0},t_{0}, \frac{(1+\varepsilon)k-s}{\rho} \right)}
\\
&=\gamma\rho^{2} \int\limits_{0}^{(v_{-}-k)_{-}}
\frac{ds}{((1+\varepsilon)k-s)\,\psi\left( x_{0},t_{0}, \frac{(1+\varepsilon)k-s}{\rho} \right) }.
\end{aligned}
$$

If $c_{\ast}\geqslant \varepsilon^{-j_{\ast}}$, then by \eqref{eq5.1}
$$
\frac{k}{\rho}=\frac{\varepsilon^{j}\omega\lambda(\rho)}{\rho}
\geqslant \varepsilon^{j_{\ast}}c_{\ast}(1+b_{0})\rho^{-\delta/\overline{\delta}}
\geqslant (1+b_{0})R_{0}^{-\delta},
$$
so condition  ${\rm (\Psi_{2})}$ is applicable and by ${\rm (\Psi_{2})}$
$$
\psi\left( x_{0},t_{0}, \frac{(1+\varepsilon)k-s}{\rho} \right)
\geqslant
\psi\left( x_{0},t_{0}, \frac{(1+\varepsilon)k}{\rho} \right)
\geqslant \frac{\psi\left( x_{0},t_{0}, k/\rho \right)}{1+\varepsilon},
$$
provided that $0<s<(v_{-}-k)_{-}$. Therefore,
$$
\begin{aligned}
\int\limits_{0}^{(v_{-}-k)_{-}}\frac{(1+\varepsilon)k-s}
{\mathcal{G}\left( x_{0},t_{0}, \frac{(1+\varepsilon)k-s}{\rho} \right)}\,dx
&\leqslant
\frac{\gamma(1+\varepsilon)\rho^{2}}{\psi\left( x_{0},t_{0}, k/\rho \right)}
\int\limits_{0}^{(v_{-}-k)_{-}} \frac{ds}{(1+\varepsilon)k-s}
\\
&\leqslant
\frac{\gamma\rho^{2}}{\psi\left( x_{0},t_{0}, k/\rho \right)}\,
\ln\frac{(1+\varepsilon)k}{(1+\varepsilon)k-(v_{-}-k)_{-}},
\quad k=\varepsilon^{j}\omega\lambda(\rho).
\end{aligned}
$$
From this, we estimate the first term on the right-hand side of \eqref{eq2.5} as
follows:
\begin{equation}\label{eq5.4}
\begin{aligned}
&\frac{\gamma}{\theta}\int\limits_{B_{\rho}(x_{0})\times\{t\}}
\Phi_{k}(x_{0},t_{0}, v_{-})\,\zeta^{\,c_{2}}\,dx
\\
&\leqslant
\gamma\, \frac{\psi\left(x_{0},t_{0}, \eta\,\frac{\omega\lambda(\rho)}{\rho}\right)}
{\psi\left(x_{0},t_{0}, \varepsilon^{j}\,\frac{\omega\lambda(\rho)}{\rho}\right)}\,
\int\limits_{B_{\rho}(x_{0})\times\{t\}}
\ln\frac{(1+\varepsilon)\varepsilon^{j}\omega\lambda(\rho)}
{(1+\varepsilon)\varepsilon^{j}\omega\lambda(\rho)-(v_{-}-\varepsilon^{j}\omega\lambda(\rho))_{-}}\,
\zeta^{\,c_{2}}\,dx
\\
&\leqslant \frac{\gamma}{\eta}
\int\limits_{B_{\rho}(x_{0})\times\{t\}}
\ln\frac{(1+\varepsilon)\varepsilon^{j}\omega\lambda(\rho)}
{(1+\varepsilon)\varepsilon^{j}\omega\lambda(\rho)-(v_{-}-\varepsilon^{j}\omega\lambda(\rho))_{-}}\,
\zeta^{\,c_{2}}\,dx
\\
&\leqslant \gamma
\int\limits_{B_{\rho}(x_{0})\times\{t\}}
\ln\frac{(1+\varepsilon)\varepsilon^{j}\omega\lambda(\rho)}
{(1+\varepsilon)\varepsilon^{j}\omega\lambda(\rho)-(v_{-}-\varepsilon^{j}\omega\lambda(\rho))_{-}}\,
\,\frac{t-t_{0}+\theta}{\theta}\,\zeta^{\,c_{2}}\,dx,
\end{aligned}
\end{equation}
provided that $t\in (t_{0}-\theta/2, t_{0})$ and $\varepsilon\in(0,\eta)$.

Using \eqref{eq5.2}, by the Poincar\'{e} inequality we estimate the second term on the left-hand
side of \eqref{eq2.5} as follows:
\begin{equation}\label{eq5.5}
\begin{aligned}
&\gamma^{-1}\varepsilon_{1}^{1-p}\rho\int\limits_{B_{\rho}(x_{0})\times\{t\}}
\left| \nabla \ln \frac{(1+\varepsilon)k}{w_{k,\varepsilon}} \right|
\frac{g(x,t,w_{k,\varepsilon}/\rho)}{g(x_{0},t_{0},w_{k,\varepsilon}/\rho)}\,
\frac{t-t_{0}+\theta}{\theta}\,\zeta^{\,c_{2}}\,dx
\\
&\geqslant \gamma^{-1}\varepsilon_{1}^{1-p}\int\limits_{B_{\rho}(x_{0})\times\{t\}}
\ln \frac{(1+\varepsilon)\varepsilon^{j}\omega\lambda(\rho)}
{(1+\varepsilon)\varepsilon^{j}\omega\lambda(\rho)-(v_{-}-\varepsilon^{j}\omega\lambda(\rho))_{-}}\,
\frac{t-t_{0}+\theta}{\theta}\,\zeta^{\,c_{2}}\,dx.
\end{aligned}
\end{equation}

Collecting estimates \eqref{eq5.3}--\eqref{eq5.5} and choosing $\varepsilon_{1}$ from the
condition $\gamma^{-1}\varepsilon_{1}^{1-p}=2\gamma$, we rewrite \eqref{eq2.5} in the
following form:
\begin{equation}\label{eq5.6}
\begin{aligned}
&D^{-}\int\limits_{B_{\rho}(x_{0})\times\{t\}}
\Phi_{\varepsilon^{j}\omega\lambda(\rho)}(x_{0},t_{0},v_{-})\,
\frac{t-t_{0}+\theta}{\theta}\,\zeta^{\,c_{2}}\,dx
\\
&+\gamma^{-1}\int\limits_{B_{\rho}(x_{0})\times\{t\}}
\ln \frac{(1+\varepsilon)\varepsilon^{j}\omega\lambda(\rho)}
{(1+\varepsilon)\varepsilon^{j}\omega\lambda(\rho)-(v_{-}-\varepsilon^{j}\omega\lambda(\rho))_{-}}\,
\frac{t-t_{0}+\theta}{\theta}\,\zeta^{\,c_{2}}\,dx
\\
&\leqslant \gamma |B_{\rho}(x_{0})|
\quad \quad \text{for all} \ t\in(t_{0}-\theta/2, t_{0}).
\end{aligned}
\end{equation}

\textsl{Claim}. Set $f(x,t,{\rm w}):=\dfrac{{\rm w}}{\mathcal{G}(x,t,{\rm w})}$, ${\rm w}>0$.
The following inequalities hold:
\begin{equation}\label{eq5.7}
\int\limits_{0}^{\sigma}f\left(x_{0},t_{0},
\frac{(\varepsilon+s)\varepsilon^{j}\omega\,\lambda(\rho)}
{\rho} \right)ds
\geqslant \sigma\int\limits_{0}^{1}
f\left(x_{0},t_{0}, \frac{(\varepsilon+s)\varepsilon^{j}\omega\,\lambda(\rho)}
{\rho} \right)ds, \quad \sigma\in(0,1),
\end{equation}
\begin{equation}\label{eq5.8}
\int\limits_{0}^{\varepsilon}
f\left(x_{0},t_{0}, \frac{(\varepsilon+s)\varepsilon^{j}\omega\,\lambda(\rho)}
{\rho} \right)ds \leqslant
\frac{2q\ln2}{2q\ln2+\ln\frac{1}{2\varepsilon}}
\int\limits_{0}^{1}
f\left(x_{0},t_{0}, \frac{(\varepsilon+s)\varepsilon^{j}\omega\,\lambda(\rho)}
{\rho} \right)ds.
\end{equation}

Indeed, inequality \eqref{eq5.7} is a consequence of the fact that the function
$f(x_{0},t_{0}, {\rm v})$ is non-increasing,
\begin{multline*}
\int\limits_{0}^{\sigma}f\left(x_{0},t_{0},
\frac{(\varepsilon+s)\varepsilon^{j}\omega\,\lambda(\rho)}
{\rho} \right)ds
\\
=\sigma\int\limits_{0}^{1}f\left(x_{0},t_{0},
\frac{(\varepsilon+\sigma s)\varepsilon^{j}\omega\,\lambda(\rho)}
{\rho} \right)ds
\geqslant
\sigma\int\limits_{0}^{1}f\left(x_{0},t_{0},
\frac{(\varepsilon+s)\varepsilon^{j}\omega\,\lambda(\rho)}
{\rho} \right)ds.
\end{multline*}
We note that inequality \eqref{eq5.8} in the case $q\leqslant2$ was proved in
\cite[inequality~7.12]{SkrVoitNA20}.
First we observe that by ${\rm (g_{1})}$, ${\rm (\Psi_{2})}$
\begin{equation}\label{eq5.9}
\begin{aligned}
f(x_{0},t_{0}, \gamma{\rm w})
&=\frac{\gamma{\rm w}}{\mathcal{G}(x_{0},t_{0}, \gamma{\rm w})}
\geqslant \frac{1}{g(x_{0},t_{0}, \gamma{\rm w})}
=\frac{1}{\gamma{\rm w}}\,\frac{1}{\psi(x_{0},t_{0}, \gamma{\rm w})}
\\
&\geqslant \frac{1}{\gamma{\rm w}}\,\frac{1}{\psi(x_{0},t_{0}, {\rm w})}
=\frac{1}{\gamma g(x_{0},t_{0}, {\rm w})} \geqslant
\frac{{\rm w}}{\gamma q\, \mathcal{G}(x_{0},t_{0}, {\rm w})}=
\frac{1}{\gamma q}\,f(x_{0},t_{0}, {\rm w}),
\end{aligned}
\end{equation}
provided that $\gamma\geqslant1$ and ${\rm w}\geqslant (1+b_{0})R_{0}^{-\delta}$.

Fix $\mathcal{J}=\mathcal{J}(\varepsilon)$ by the condition
$2^{-\mathcal{J}-1}<\varepsilon\leqslant 2^{-\mathcal{J}}$, then we have
\begin{multline*}
\int\limits_{0}^{1}f\left(x_{0},t_{0},
\frac{(\varepsilon+s)\varepsilon^{j}\omega\,\lambda(\rho)}
{\rho} \right)ds
\\
\geqslant
\int\limits_{0}^{\varepsilon}f\left(x_{0},t_{0},
\frac{(\varepsilon+s)\varepsilon^{j}\omega\,\lambda(\rho)}
{\rho} \right)ds+
\sum\limits_{i=0}^{\mathcal{J}-1}\int\limits_{2^{i}\varepsilon}^{2^{i+1}\varepsilon}
f\left(x_{0},t_{0},
\frac{(\varepsilon+s)\varepsilon^{j}\omega\,\lambda(\rho)}
{\rho} \right)ds.
\end{multline*}
To estimate the second term on the right-hand side of the last inequality
we use \eqref{eq5.9}, we have
$$
\begin{aligned}
&\int\limits_{2^{i}\varepsilon}^{2^{i+1}\varepsilon}
f\left(x_{0},t_{0},
\frac{(\varepsilon+s)\varepsilon^{j}\omega\,\lambda(\rho)}
{\rho} \right)ds= 2^{i}\int\limits_{0}^{\varepsilon}
f\left(x_{0},t_{0},
\frac{(2^{i}s+(2^{i}+1) \varepsilon)\varepsilon^{j}\omega\,\lambda(\rho)}
{\rho} \right)ds
\\
&\geqslant 2^{i}\int\limits_{0}^{\varepsilon}
f\left(x_{0},t_{0},
\frac{(2^{i}+1)(s+ \varepsilon)\varepsilon^{j}\omega\,\lambda(\rho)}
{\rho} \right)ds \geqslant \frac{2^{i}}{q(2^{i}+1)}
\int\limits_{0}^{\varepsilon}
f\left(x_{0},t_{0},
\frac{(s+ \varepsilon)\varepsilon^{j}\omega\,\lambda(\rho)}
{\rho} \right)ds
\\
&\geqslant \frac{1}{2q}
\int\limits_{0}^{\varepsilon}
f\left(x_{0},t_{0},
\frac{(s+ \varepsilon)\varepsilon^{j}\omega\,\lambda(\rho)}
{\rho} \right)ds, \quad i=0,1,\ldots, \mathcal{J}-1.
\end{aligned}
$$
Collecting last two inequalities, we obtain
$$
\int\limits_{0}^{1}f\left(x_{0},t_{0},
\frac{(\varepsilon+s)\varepsilon^{j}\omega\,\lambda(\rho)}
{\rho} \right)ds\geqslant \left(1+\frac{\mathcal{J}}{2q}\right)
\int\limits_{0}^{\varepsilon}
f\left(x_{0},t_{0},
\frac{(s+ \varepsilon)\varepsilon^{j}\omega\,\lambda(\rho)}
{\rho} \right)ds,
$$
from which the required \eqref{eq5.8} follows.
This proves the claim.

Fix $\overline{t}\in(t_{0}-\theta/2, t_{0})$ such that
$Y_{j+1}(\overline{t})=y_{j+1}$. If
$$
D^{-}\int\limits_{B_{\rho}(x_{0})\times\{\overline{t}\}}
\Phi_{\varepsilon^{j}\omega\,\lambda(\rho)}(x_{0},t_{0}, v_{-})\,
\frac{\overline{t}-t_{0}+\theta}{\theta}\,\zeta^{\,c_{2}}\,dx
\geqslant0,
$$
then inequality \eqref{eq5.6} implies that
\begin{equation}\label{eq5.10}
y_{j+1}\ln\frac{1}{2\varepsilon}\leqslant \gamma.
\end{equation}
For fixed $\nu\in(0,1)$ we choose $\varepsilon$ from the condition
$\varepsilon\leqslant\dfrac{1}{2}\exp\left(-\dfrac{\gamma}{\nu}\right)$, then
\eqref{eq5.10} yields
\begin{equation}\label{eq5.11}
y_{j+1}\leqslant \nu.
\end{equation}

Assume now that
$$
D^{-}\int\limits_{B_{\rho}(x_{0})\times\{\overline{t}\}}
\Phi_{\varepsilon^{j}\omega\,\lambda(\rho)}(x_{0},t_{0}, v_{-})\,
\frac{\overline{t}-t_{0}+\theta}{\theta}\,\zeta^{\,c_{2}}\,dx
<0.
$$
Define
$$
t_{\ast}:=\sup\left\{ t\in\left(t_{0}-\theta/2,\, t_{0}\right):
D^{-}\int\limits_{B_{\rho}(x_{0})\times\{t\}}
\Phi_{\varepsilon^{j}\omega\,\lambda(\rho)}(x_{0},t_{0}, v_{-})\,
\frac{t-t_{0}+\theta}{\theta}\,\zeta^{\,c_{2}}\,dx\geqslant0  \right\},
$$
then we obtain
\begin{equation}\label{eq5.12}
I(\overline{t}):=\int\limits_{B_{\rho}(x_{0})\times\{\overline{t}\}}
\Phi_{\varepsilon^{j}\omega\,\lambda(\rho)}(x_{0},t_{0}, v_{-})\,
\frac{\overline{t}-t_{0}+\theta}{\theta}\,\zeta^{\,c_{2}}\,dx
\leqslant I(t_{\ast}).
\end{equation}
By \eqref{eq5.8} we have
\begin{equation}\label{eq5.13}
\begin{aligned}
I(\overline{t})
&\geqslant
\int\limits_{A_{j+1}(\overline{t})}\frac{\overline{t}-t_{0}+\theta}{\theta}\,\zeta^{\,c_{2}}\,dx
\int\limits_{0}^{\varepsilon^{j}(1-\varepsilon)\omega\lambda(\rho)}
\frac{(1+\varepsilon)\varepsilon^{j}\omega\lambda(\rho)-s}
{\mathcal{G} \left( x_{0},t_{0}, \frac{(1+\varepsilon)\varepsilon^{j}\omega\lambda(\rho)-s}
{\rho}  \right)}\,ds
\\
&=\varepsilon^{j}\omega\lambda(\rho)
\int\limits_{A_{j+1}(\overline{t})}\frac{\overline{t}-t_{0}+\theta}{\theta}\,\zeta^{\,c_{2}}\,dx
\int\limits_{0}^{1-\varepsilon} f\left( x_{0},t_{0},
\frac{1+\varepsilon-s}{\rho}\,\varepsilon^{j}\omega\lambda(\rho)  \right)ds
\\
&=\varepsilon^{j}\omega\lambda(\rho)
\int\limits_{A_{j+1}(\overline{t})}\frac{\overline{t}-t_{0}+\theta}{\theta}\,\zeta^{\,c_{2}}\,dx
\int\limits_{\varepsilon}^{1} f\left( x_{0},t_{0},
\frac{\varepsilon+s}{\rho}\,\varepsilon^{j}\omega\lambda(\rho)  \right)ds
\\
&\geqslant y_{j+1} \left( 1-\frac{2q\ln2}{2q\ln2+\ln\frac{1}{2\varepsilon}} \right)
\varepsilon^{j}\omega\lambda(\rho)\,|B_{\rho}(x_{0})|
\int\limits_{0}^{1} f\left( x_{0},t_{0},
\frac{\varepsilon+s}{\rho}\,\varepsilon^{j}\omega\lambda(\rho)  \right)ds.
\end{aligned}
\end{equation}

Let us estimate the term on the right-hand side of inequality \eqref{eq5.12}.
By Fubini's theorem we conclude that
\begin{multline*}
I(t_{\ast})\leqslant \int\limits_{B_{\rho}(x_{0})\times\{t_{\ast}\}}
\frac{t_{\ast}-t_{0}+\theta}{\theta}\,\zeta^{\,c_{2}}\,dx
\int\limits_{0}^{\varepsilon^{j}\omega\lambda(\rho)}
\frac{(1+\varepsilon)\varepsilon^{j}\omega\lambda(\rho)-s}
{\mathcal{G} \left( x_{0},t_{0}, \frac{(1+\varepsilon)\varepsilon^{j}\omega\lambda(\rho)-s}
{\rho}  \right)}\,
\mathbb{I}_{\{\varepsilon^{j}\omega\lambda(\rho)-v_{-}>\varepsilon^{j}\omega\lambda(\rho)s\}}
\,ds
\\
= \varepsilon^{j}\omega\lambda(\rho)\int\limits_{0}^{1}
f\left(x_{0},t_{0}, \frac{1+\varepsilon-s}{\rho}\,\varepsilon^{j}\omega\lambda(\rho)\right)ds
\int\limits_{B_{\rho}(x_{0})\times\{t_{\ast}\}}
\frac{t_{\ast}-t_{0}+\theta}{\theta}\,
\mathbb{I}_{\{\varepsilon^{j}\omega\lambda(\rho)-v_{-}>\varepsilon^{j}\omega\lambda(\rho)s\}}\,
\zeta^{\,c_{2}}\,dx.
\end{multline*}
Similarly to \eqref{eq5.10} we obtain
$$
\int\limits_{B_{\rho}(x_{0})\times\{t_{\ast}\}}
\frac{t_{\ast}-t_{0}+\theta}{\theta}\,\zeta^{\,c_{2}}\,
\mathbb{I}_{\{\varepsilon^{j}\omega\lambda(\rho)-v_{-}>\varepsilon^{j}\omega\lambda(\rho)\}}\,
dx\leqslant \frac{\gamma}{\ln \frac{1+\varepsilon}{1+\varepsilon-s}}\,
|B_{\rho}(x_{0})|,
$$
for any $s\in(0,1)$. Particularly, if
$(1+\varepsilon)\left( 1-\exp\left( -\dfrac{2\gamma}{\nu}\right)\right)<s<1$,
then
$$
\int\limits_{B_{\rho}(x_{0})\times\{t_{\ast}\}}
\frac{t_{\ast}-t_{0}+\theta}{\theta}\,\zeta^{\,c_{2}}\,
\mathbb{I}_{\{\varepsilon^{j}\omega\lambda(\rho)-v_{-}>\varepsilon^{j}\omega\lambda(\rho)\}}\,
dx\leqslant \frac{\nu}{2}\,|B_{\rho}(x_{0})|.
$$
Choosing $s_{\ast}$ from the condition
$s_{\ast}:=(1+\varepsilon)\left( 1-\exp\left( -\dfrac{4\gamma}{\nu}\right)\right)$,
and assuming that $y_{j+1}\geqslant\nu$, from the previous and \eqref{eq5.7} we obtain
\begin{equation}\label{eq5.14}
\begin{aligned}
I(t_{\ast})
&\leqslant y_{j}\,\varepsilon^{j}\omega\lambda(\rho)\,|B_{\rho}(x_{0})|
\\
&\times\left(
\int\limits_{0}^{s_{\ast}}
f\left(x_{0},t_{0}, \frac{1+\varepsilon-s}{\rho}\,\varepsilon^{j}\omega\lambda(\rho)\right)ds
+
\int\limits_{s_{\ast}}^{1}
f\left(x_{0},t_{0}, \frac{1+\varepsilon-s}{\rho}\,\varepsilon^{j}\omega\lambda(\rho)\right)ds
\right)
\\
&\leqslant y_{j}\,\varepsilon^{j}\omega\lambda(\rho)\,|B_{\rho}(x_{0})|
\\
&\times
\left(
\int\limits_{0}^{1}
f\left(x_{0},t_{0}, \frac{\varepsilon+s}{\rho}\,\varepsilon^{j}\omega\lambda(\rho)\right)ds
-\frac{1}{2}
\int\limits_{0}^{1-s_{\ast}}
f\left(x_{0},t_{0}, \frac{\varepsilon+s}{\rho}\,\varepsilon^{j}\omega\lambda(\rho)\right)ds
\right)
\\
&\leqslant y_{j}\left(1-\frac{1-s_{\ast}}{2} \right)
\varepsilon^{j}\omega\lambda(\rho)\,|B_{\rho}(x_{0})|
\int\limits_{0}^{1}
f\left(x_{0},t_{0}, \frac{\varepsilon+s}{\rho}\,\varepsilon^{j}\omega\lambda(\rho)\right)ds.
\end{aligned}
\end{equation}
Collecting \eqref{eq5.13} and \eqref{eq5.14} we obtain that either
\begin{equation}\label{eq5.15}
y_{j}\leqslant \nu,
\end{equation}
or
\begin{equation}\label{eq5.16}
y_{j+1}\leqslant \left(1-\frac{1-s_{\ast}}{2} \right)
\left(1-\frac{2q\ln2}{2q\ln2+\ln\frac{1}{2\varepsilon}} \right)^{-1}
y_{j}.
\end{equation}
By our choice of $s_{\ast}$ we have for sufficiently small $\varepsilon>0$
\begin{multline*}
\left(1-\frac{1-s_{\ast}}{2} \right)
\left(1-\frac{2q\ln2}{2q\ln2+\ln\frac{1}{2\varepsilon}} \right)^{-1}
\\
\leqslant 1-\frac{1}{2}
\left( \exp\left(-\frac{4\gamma}{\nu}\right)-\varepsilon -
\frac{2q\ln2}{2q\ln2+\ln\frac{1}{2\varepsilon}} \right)
\leqslant 1-\frac{1}{2}
\left( \exp\left(-\frac{4\gamma}{\nu}\right)-
\frac{4q\ln2}{2q\ln2+\ln\frac{1}{2\varepsilon}} \right).
\end{multline*}
Choosing $\varepsilon$ so small that
$\dfrac{4q\ln2}{2q\ln2+\ln\frac{1}{2\varepsilon}}=\dfrac{1}{2}\exp\left(-\dfrac{4\gamma}{\nu}\right)$,
we obtain from \eqref{eq5.16} that
$$
y_{j+1}\leqslant
\left( 1-\frac{1}{4} \exp\left(-\frac{4\gamma}{\nu}\right)\right)
y_{j}, \quad j=0,1,2, \ldots\,.
$$
Iterating this inequality, we obtain
$$
y_{j_{\ast}}\leqslant
\left( 1-\frac{1}{4} \exp\left(-\frac{4\gamma}{\nu}\right)\right)^{j_{\ast}}
y_{0}\leqslant
\left( 1-\frac{1}{4} \exp\left(-\frac{4\gamma}{\nu}\right)\right)^{j_{\ast}}.
$$
Choosing $j_{\ast}$ so large that
$
\left( 1-\dfrac{1}{4} \exp\left(-\dfrac{4\gamma}{\nu}\right)\right)^{j_{\ast}}
\leqslant\nu
$, we arrive at
\begin{equation}\label{eq5.17}
y_{j_{\ast}}\leqslant \nu.
\end{equation}
Using Lemma \ref{lem3.5} in the cylinder $Q_{\rho,\theta'}(x_{0},t_{0})$,
$\theta':=\dfrac{\rho^{2}}{\psi\left(x_{0},t_{0},
\frac{\varepsilon^{j_{\ast}}\omega\lambda(\rho)}{\rho}\right)}$
and assuming that $c_{\ast}\geqslant 4\varepsilon^{-1-j_{\ast}}$,
by \eqref{eq5.1} and \eqref{eq5.17} we obtain that
$$
v_{-}(x,t)\geqslant \frac{1}{4}\,\varepsilon^{j_{\ast}}\omega\lambda(\rho)
\quad \text{for a.a.} \  (x,t)\in Q_{\rho/2,\theta'/2}(x_{0},t_{0}),
$$
which implies that
\begin{equation}\label{eq5.18}
\osc\limits_{Q_{\rho/2,\theta'/2}(x_{0},t_{0})}u
\leqslant
\left(1-\frac{1}{4}\,\varepsilon^{j_{\ast}}\omega\lambda(\rho)\right)\omega.
\end{equation}

For $i=0,1,2,\ldots$ define the sequences $r_{i}:=c^{-i}\rho$, $\omega_{0}:=\omega$,
$$
\omega_{i+1}:=
\max\left\{ \left(1-\frac{1}{4}\,\varepsilon^{j_{\ast}}\lambda(r_{i})\right)\omega_{i},\,
c_{\ast}(1+b_{0})\,\frac{r_{i}^{1-\frac{\delta}{\overline{\delta}}}}{\lambda(r_{i})} \right\},
$$
$$
\theta_{i}:= \frac{r_{i}^{2}}{\psi\left(x_{0},t_{0},
\dfrac{\omega_{i}\lambda(r_{i})}{r_{i}}\right)},
\quad
Q_{i}:=Q_{r_{i}, \theta_{i}}(x_{0},t_{0}),
$$
where $c>1$ will be chosen depending on the known data only.

If inequality \eqref{eq5.1} holds, then comletely similar to that of Section \ref{Sect4}
we have
$$
\omega_{i+1}\geqslant c_{\ast}(1+b_{0})
\,\frac{r_{i+1}^{1-\frac{\delta}{\overline{\delta}}}}{\lambda(r_{i+1})}
\left( \frac{1}{2}\,c^{\delta_{0}-\frac{\delta}{\overline{\delta}}}
\bigg(1-\frac{1}{4}\,\varepsilon^{j_{\ast}}\bigg) \right)^{i+1},
$$
so, if
$c^{\delta_{0}-\frac{\delta}{\overline{\delta}}}
\geqslant 2\left(1-\dfrac{1}{4}\,\varepsilon^{j_{\ast}} \right)^{-1}$,
from the previous we obtain
$$
\omega_{i+1}\geqslant c_{\ast}(1+b_{0})
\,\frac{r_{i+1}^{1-\frac{\delta}{\overline{\delta}}}}{\lambda(r_{i+1})},
\quad i=0,1,2,\ldots\,.
$$
Moreover, by \eqref{eq4.11} we also have $\theta_{i+1}\leqslant \theta_{i}$,
$i=0,1,2,\ldots$\,. Inequality \eqref{eq5.18} implies that
$\osc\limits_{Q_{1}}u\leqslant\omega_{1}$. Repeating the previous procedure,
similarly to \eqref{eq4.12} we obtain for any $i\geqslant1$
\begin{equation}\label{eq5.19}
\osc\limits_{Q_{i}}u\leqslant\omega
\exp\left( -\gamma \int\limits_{r_{i}}^{c\rho}\lambda(s)\,\frac{ds}{s}
\right)+\gamma(1+b_{0})\,\frac{\rho^{1-\delta_{0}}}{\lambda(\rho_{0})}.
\end{equation}

By ${\rm (g_{1})}$ we have the inclusion $\widetilde{Q}_{i}\subset Q_{i}$,
$\widetilde{Q}_{i}:=B_{r_{i}}(x_{0})\times(t_{0}-\widetilde{\theta}_{i}, t_{0})$,
$\widetilde{\theta}_{i}:=\dfrac{r_{i}(c_{\ast}(1+b_{0}))^{2-q}}{g(x_{0},t_{0},1)}$.
Therefore, inequality \eqref{eq5.19} implies
$$
\osc\limits_{\widetilde{Q}_{i}}
\leqslant
2M \exp\left( -\gamma \int\limits_{r_{i}}^{c\rho}\lambda(s)\,\frac{ds}{s}
\right)+\gamma(1+b_{0})\,\frac{\rho^{1-\delta_{0}}}{\lambda(\rho_{0})},
$$
which implies the continuity of $u$ in the ''singular'' case.
To complete the proof of Theorem \ref{th2.1}, note that in the case $q\leqslant2$,
as it was mentioned in Section \ref{Introduction}, by (${\rm g}_{1}$) condition ($\Psi_{2}$)
holds with $\delta=b_{0}=0$; therefore,
the number $R_{0}$ claimed in the definition of the cylinder $Q_{R_{0},R_{0}}(x_{0},t_{0})$
depends only on the distance between $(x_{0},t_{0})$ and $\partial\Omega_{T}$, so, in this case
$u\in C_{{\rm loc}}(\Omega_{T})$.

\vskip3.5mm
{\bf Acknowledgements.} The research of the first author was supported by grants of Ministry of Education and Science of Ukraine
(project numbers are 0118U003138, 0119U100421).

\bigskip

CONTACT INFORMATION

\medskip

\medskip
Igor I.~Skrypnik\\Institute of Applied Mathematics and Mechanics,
National Academy of Sciences of Ukraine, Gen. Batiouk Str. 19, 84116 Sloviansk, Ukraine\\
Vasyl' Stus Donetsk National University,
600-richcha Str. 21, 21021 Vinnytsia, Ukraine\\iskrypnik@iamm.donbass.com

\medskip
Mykhailo V.~Voitovych\\Institute of Applied Mathematics and Mechanics,
National Academy of Sciences of Ukraine, Gen. Batiouk Str. 19, 84116 Sloviansk, Ukraine\\voitovichmv76@gmail.com


\begin{thebibliography}{99}


%\bibitem{Alhutov97}
%Yu.\,A.~Alkhutov,
%The Harnack inequality and the H\"{o}lder property of solutions of nonlinear elliptic equations
%with a nonstandard growth condition (Russian),
%Differ. Uravn. \textbf{33} (1997), no.~12, 1651--1660;
%translation in Differential Equations \textbf{33} (1997), no.~12, 1653--1663 (1998).



\bibitem{AlhutovKrash08}
Yu.\,A.~Alkhutov, O.\,V.~Krasheninnikova,
On the continuity of solutions of elliptic equations with a variable order of nonlinearity,
(Russian)
Tr. Mat. Inst. Steklova \textbf{261} (2008), Differ. Uravn. i Din. Sist., 7--15;
translation in Proc. Steklov Inst. Math. \textbf{261} (2008), no.~1--10.



\bibitem{AlkhSurnAlgAn19}
Yu.\,A.~Alkhutov, M.\,D.~Surnachev,
Behavior at a boundary point of solutions of the Dirichlet problem for the $p(x)$-Laplacian,
(Russian) Algebra i Analiz \textbf{31} (2019), no. 2, 88--117;
translation in St. Petersburg Math. J. \textbf{31} (2020), no. 2, 251--271.

\bibitem{ZhikAlkhTSP11}
Yu.\,A.~Alkhutov, V.\,V.~Zhikov,
H\"{o}lder continuity of solutions of parabolic equations with variable nonlinearity exponent,
Translation of Tr. Semin. im. I. G. Petrovskogo No. 28 (2011), Part I, 8--74;
J. Math. Sci. (N.Y.) \textbf{179} (2011), no. 3, 347--389.



\bibitem{AntDiazShm2002monogr}
S.\,N.~Antontsev, J.\,I.~D\'{\i}az, S.~Shmarev,
Energy Methods for Free Boundary Problems. Applications to Nonlinear PDEs and Fluid Mechanics,
in: Progress in Nonlinear Differential Equations and their Applications, vol. 48,
Birkhauser Boston, Inc., Boston, MA, 2002.


\bibitem{AntZhikov2005}
S.~Antontsev, V.~Zhikov,
Higher integrability for parabolic equations of $p(x,t)$-Laplacian type,
 Adv. Differential Equations \textbf{10} (2005), no. 9, 1053--1080.



\bibitem{BarBog2014}
P. Baroni, V. B\"{o}gelein,
Calder\'{o}n-Zygmund estimates for parabolic $p(x, t)$-Laplacian systems,
Rev. Mat. Iberoam. \textbf{30} (2014), no. 4, 1355--1386.




\bibitem{BarColMing}
P.~Baroni, M.~Colombo, G.~Mingione,
Harnack inequalities for double phase functionals,
Nonlinear Anal. \textbf{121} (2015), 206--222.



\bibitem{BarColMingStPt16}
P.~Baroni, M.~Colombo, G.~Mingione,
Non-autonomous functionals, borderline cases and related
function classes,
St. Petersburg Math. J. \textbf{27} (2016), 347--379.



\bibitem{BarColMingCalc.Var.18}
P. Baroni, M. Colombo, G. Mingione,
Regularity for general functionals with double phase, Calc. Var.
Partial Differential Equations \textbf{57} (2018), Paper No. 62, 48 pp.


\bibitem{BogDuzaar2012}
V. B\"{o}gelein, F. Duzaar,
H\"{o}lder estimates for parabolic $p(x, t)$-Laplacian systems,
Math. Ann. \textbf{354} (2012), no. 3, 907--938.



\bibitem{BurchSkrPotAn}
K.\,O. Buryachenko, I.\,I. Skrypnik,
Local Continuity and Harnack’s Inequality for Double-Phase Parabolic Equations,
Potential Anal. (2020). https://doi.org/10.1007/s11118-020-09879-9



\bibitem{ColMing218}
M.~Colombo, G.~Mingione,
Bounded minimisers of double phase variational integrals,
Arch. Rational Mech. Anal.  \textbf{218} (2015), no. 1, 219--273.



\bibitem{ColMing15}
M.~Colombo, G.~Mingione, Regularity for double phase variational problems,
Arch. Rational Mech. Anal.  \textbf{215} (2015), no. 2, 443--496.



\bibitem{ColMingJFnctAn16}
M. Colombo, G. Mingione,
Calderon-Zygmund estimates and non-uniformly elliptic operators,
J. Funct. Anal. \textbf{270} (2016), 1416--1478.





\bibitem{DiBenedettoDegParEq}
E.~Di\,Benedetto, Degenerate Parabolic Equations, Springer-Verlag, New York, 1993.






\bibitem{DienHarHastRuzVarEpn}
L. Diening, P. Harjulehto, P. H\"{a}st\"{o}, M.~R\r{u}\v{z}i\v{c}ka,
Lebesgue and Sobolev Spaces with Variable Exponents,
in: Lecture Notes in Mathematics, 2017,
Springer, Heidelberg, 2011, x+509 pp.



\bibitem{DingZhangZhou2020}
Ding Mengyao, Zhang Chao, Zhou Shulin,
Global boundedness and H\"{o}lder regularity of solutions to general $p(x,t)$-Laplace parabolic equations,
 Math. Methods Appl. Sci. \textbf{43} (2020), no. 9, 5809--5831.




\bibitem{HadzhySkrVoit}
O.\,V. Hadzhy, I.\,I. Skrypnik, M.\,V. Voitovych,
Interior continuity, continuity up to the boundary and Harnack's inequality
for double-phase elliptic equations with non-logarithmic growth,
arXiv:2012.10960v1 [math.AP].



\bibitem{HarHastOrlicz}
P. Harjulehto, P. H\"{a}st\"{o},
Orlicz Spaces and Generalized Orlicz Spaces, in: Lecture Notes in Mathematics,
vol. 2236, Springer, Cham, 2019, p. X+169 http://dx.doi.org/10.1007/978-3-030-15100-3


\bibitem{HarHastLeNuorNA2010}
P. Harjulehto, P. H\"{a}st\"{o}, \'{U}t\,V. L\^{e}, M. Nuortio,
Overview of differential equations with non-standard growth,
Nonlinear Anal. \textbf{72} (2010), no. 12, 4551--4574.



\bibitem{HwangLieberman287}
S. Hwang, G.\,M. Lieberman, H\"{o}lder continuity of bounded weak solutions to generalized
parabolic $p$-Laplacian equations I: degenerate case,
Electron. J. Differential Equations, \textbf{2015} (2015), no. 287, 1--32.


\bibitem{HwangLieberman288}
S. Hwang, G.\,M. Lieberman, H\"{o}lder continuity of bounded weak solutions to generalized
parabolic $p$-Laplacian equations II: singular case,
Electron. J. Differential Equations, \textbf{2015} (2015), no. 288, 1--24.



\bibitem{LadUr}
O.\,A.~Ladyzhenskaya, N.\,N.~Ural'tseva,
Linear and quasilinear elliptic equations,
Nauka, Moscow, 1973.



\bibitem{Lieberman91}
G.\,M.~Lieberman,
The natural generalization of the natural conditions of
Ladyzhenskaya and Ural'tseva for elliptic equations,
Comm. Partial Differential Equations \textbf{16} (1991), no. 2-3, 311--361.



\bibitem{Marcellini1989}
P.~Marcellini, Regularity of minimizers of integrals of the calculus of variations with
non standard growth conditions, Arch. Rational Mech. Anal.  \textbf{105} (1989), no.~3, 267--284.



\bibitem{Marcellini1991}
P.~Marcellini, Regularity and existence of solutions of elliptic equations with $p,q$-growth conditions,
J. Differential Equations \textbf{90} (1991), no. 1, 1--30.



\bibitem{MingioneDarkSide}
G. Mingione,
Regularity of minima: an invitation to the dark side of the calculus of variations,
Appl. Math. \textbf{51} (2006), no. 4, 355--426.



\bibitem{Ruzicka2000}
M.~R\r{u}\v{z}i\v{c}ka,
Electrorheological fluids: Modeling and Mathematical Theory, in:
Lecture Notes in Mathematics, vol.~1748, Springer-Verlag, Berlin, 2000.



\bibitem{ShSkrVoit20}
M.\,A. Shan,  I.\,I. Skrypnik,  M.\,V. Voitovych,
Harnack's inequality for quasilinear elliptic equations with
generalized Orlicz growth, arxiv:2008.03744v1 [math.AP].



\bibitem{SkrVoitUMB19}
I.\,I. Skrypnik, M.\,V. Voitovych,
$\mathfrak{B}_{1}$ classes of De Giorgi, Ladyzhenskaya, and Ural'tseva
and their application to elliptic and parabolic equations with nonstandard growth,
J. Math. Sci. (N. Y.) \textbf{246} (2020) 75--109.



\bibitem{SkrVoitNA20}
I.\,I. Skrypnik, M.\,V. Voitovych,
$\mathcal{B}_{1}$ classes of De Giorgi-Ladyzhenskaya-Ural'tseva and
their applications to elliptic and parabolic equations with generalized Orlicz
growth conditions, Nonlinear Anal. \textbf{202} (2021) 112135.



\bibitem{SurnPrepr2018}
M.\,D. Surnachev,
On Harnack's inequality for $p(x)$-Laplacian (Russian),
Keldysh Institute Preprints 10.20948/prepr-2018-69, \textbf{69} (2018), 1--32.



\bibitem{VoitNA19}
M.\,V. Voitovych,
Pointwise estimates of solutions to $2m$-order quasilinear elliptic
equations with $m$-$(p, q)$ growth via Wolff potentials,
Nonlinear Anal. \textbf{181} (2019) 147--179.



\bibitem{Weickert}
J.~Weickert,
Anisotropic Diffusion in Image Processing, in:
European Consortium for Mathematics in Industry, B.G. Teubner, Stuttgart, 1998.


\bibitem{WinkZach2016}
P. Winkert, R. Zacher,
Global a priori bounds for weak solutions to quasilinear
parabolic equations with nonstandard growth,
Nonlinear Anal. \textbf{145} (2016), 1--23.



\bibitem{XuChen2006}
M. Xu, Y. Chen,
H\"{o}lder continuity of weak solutions for parabolic equations with nonstandard growth conditions,
Acta Math. Sin. (Engl. Ser.) \textbf{22}(3) (2006) 793--806.


\bibitem{Yao2014}
F. Yao,
H\"{o}lder regularity of the gradient for the non-homogeneous parabolic $p(x, t)$-Laplacian equations,
Math. Methods Appl. Sci., \textbf{37} (2014), no. 12, 1863--1872.



\bibitem{Yao2015}
F. Yao,
H\"{o}lder regularity for the general parabolic $p(x, t)$-Laplacian equations,
NoDEA Nonlinear Differential Equations Appl. \textbf{22} (2015), no. 1, 105--119.



\bibitem{ZhZhouXueNonAn2014}
C. Zhang, S. Zhou, X. Xue,
Global gradient estimates for the parabolic $p(x, t)$-Laplacian equation,
Nonlinear Anal. \textbf{105} (2014) 86--101.


\bibitem{ZhikIzv1983}
V.\,V.~Zhikov, Questions of convergence, duality and averaging for functionals
of the calculus of variations, (Russian)
Izv. Akad. Nauk SSSR Ser. Mat. \textbf{47} (1983), no.~5, 961--998.



\bibitem{ZhikIzv1986}
V.\,V.~Zhikov,
Averaging of functionals of the calculus of variations and elasticity theory,
(Russian)
Izv. Akad. Nauk SSSR Ser. Mat. \textbf{50} (1986), no.~4, 675--710, 877.



\bibitem{ZhikJMathPh94}
V.\,V.~Zhikov,
On Lavrentiev's phenomenon,
Russian J. Math. Phys. \textbf{3} (1995), no.~2, 249--269.



\bibitem{ZhikJMathPh9798}
V.\,V.~Zhikov,
On some variational problems,
Russian J. Math. Phys. \textbf{5} (1997), no. 1, 105--116 (1998).



\bibitem{ZhikPOMI04}
V.\,V. Zhikov,
On the density of smooth functions in Sobolev-Orlicz spaces,
(Russian) Zap. Nauchn. Sem. S.-Peterburg. Otdel. Mat. Inst. Steklov. (POMI) \textbf{310} (2004),
Kraev. Zadachi Mat. Fiz. i Smezh. Vopr. Teor. Funkts. 35 [34], 67--81, 226;
translation in J. Math. Sci. (N.Y.) \textbf{132} (2006), no. 3, 285--294.





\bibitem{ZhikKozlOlein94}
V.\,V.~Zhikov, S.\,M.~Kozlov, O.\,A.~Oleinik,
Homogenization of differential operators and integral functionals,
Springer-Verlag, Berlin, 1994.



\bibitem{ZhikPast2008MatSb}
V.\,V.~Zhikov, S.\,E.~Pastukhova,
On the improved integrability of the gradient of solutions of elliptic equations with a variable nonlinearity exponent,
(Russian)
Mat. Sb. \textbf{199} (2008), no.~12, 19--52;
translation in Sb. Math. \textbf{199} (2008), no. 11--12, 1751--1782.


\bibitem{ZhikPast2010MatNotes}
V.\,V.~Zhikov, S.\,E.~Pastukhova,
On the property of higher integrability for parabolic systems of variable order of nonlinearity,
(Russian) Mat. Zametki \textbf{87} (2010), no. 2, 179--200;
translation in Math. Notes \textbf{87} (2010), no. 1--2, 169--188.




\end{thebibliography}
\end{document}